\PassOptionsToPackage{usenames,dvipsnames}{xcolor}
\documentclass[11pt,reqno,oneside]{amsart}
\usepackage[dvipsnames]{xcolor}
\definecolor{kanzou}{RGB}{255, 137, 54}

\usepackage[citecolor=PineGreen,colorlinks=true,linkcolor=RoyalBlue]{hyperref}
\hypersetup{
pdftitle={Expected hitting time estimates on finite graphs},
pdfsubject={Mathematics, probability},
pdfauthor={L. Saloff-Coste and Y. Wang},
pdfkeywords={Hitting time, random walks, Green’s function, heat kernel estimates}
}

\usepackage{enumitem,bm}
\usepackage[all,arc]{xy}
\usepackage{mathrsfs,graphicx,subcaption,sidecap,mathtools,xparse}
\usepackage{setspace}

\usepackage[mathlines]{lineno}
\usepackage{etoolbox} 

\newcommand*\linenomathpatch[1]{%
  \expandafter\pretocmd\csname #1\endcsname {\linenomath}{}{}%
  \expandafter\pretocmd\csname #1*\endcsname{\linenomath}{}{}%
  \expandafter\apptocmd\csname end#1\endcsname {\endlinenomath}{}{}%
  \expandafter\apptocmd\csname end#1*\endcsname{\endlinenomath}{}{}%
}
\newcommand*\linenomathpatchAMS[1]{%
  \expandafter\pretocmd\csname #1\endcsname {\linenomathAMS}{}{}%
  \expandafter\pretocmd\csname #1*\endcsname{\linenomathAMS}{}{}%
  \expandafter\apptocmd\csname end#1\endcsname {\endlinenomath}{}{}%
  \expandafter\apptocmd\csname end#1*\endcsname{\endlinenomath}{}{}%
}
\expandafter\ifx\linenomath\linenomathWithnumbers
  \let\linenomathAMS\linenomathWithnumbers
  \patchcmd\linenomathAMS{\advance\postdisplaypenalty\linenopenalty}{}{}{}
\else
  \let\linenomathAMS\linenomathNonumbers
\fi

\linenomathpatch{equation*}
\linenomathpatchAMS{gather}
\linenomathpatchAMS{multline}
\linenomathpatchAMS{align}
\linenomathpatchAMS{alignat}
\linenomathpatchAMS{flalign}

\usepackage{dsfont}
\usepackage{kpfonts}
\usepackage[T1]{fontenc}
\usepackage[utf8]{inputenc}

\usepackage{microtype} 
\DeclareFontFamily{OT1}{rsfs}{}
\DeclareFontShape{OT1}{rsfs}{n}{it}{<-> rsfs10}{}
\DeclareMathAlphabet{\mathscr}{OT1}{rsfs}{n}{it}

\theoremstyle{plain} 
\newtheorem{thm}{Theorem}[section]
\newtheorem{cor}[thm]{Corollary}
\newtheorem{prop}[thm]{Proposition}
\newtheorem{lem}[thm]{Lemma}

\theoremstyle{definition}
\newtheorem{defn}[thm]{Definition}

\theoremstyle{remark}

\makeatletter
\def\paragraph{\@startsection{paragraph}{4}%
  \z@\z@{-\fontdimen2\font}%
  {\normalfont\it}}
\makeatother

\renewcommand{\and}{\qquad \text{and} \qquad}
\newcommand{\Z}{\mathds{Z}}

\newcommand{\R}{\mathds{R}}
\newcommand{\N}{\mathds{N}}

\newcommand{\E}{\mathds{E}}
\renewcommand{\P}{\mathds{P}}

\renewcommand{\phi}{\varphi}
\newcommand{\charf}{\mathds{1}}
\newcommand{\Tau}{\mathrm{T}}

\newcommand{\G}{\mathcal{G}}
\newcommand{\V}{\mathbb{V}}

\newcommand{\st}{\texttt{ST}}
\newcommand{\vi}{\texttt{Vic}}

\NewDocumentCommand{\cl}{som}{%
  \IfBooleanTF{#1}
    {\oldnormaux*{#3}}
    {\IfNoValueTF{#2}
       {\oldnormaux*{\vphantom{dq}#3}}
       {\oldnormaux[#2]{#3}}%
    }%
}

\newcommand{\cpdot}{\; \cdot \;}

\newcommand{\syst}{(\Gamma, K, \pi)}
\newcommand{\assumptions}[1]{$(\El), (\VD), (\PI(\theta))$, and $(\CS(\theta))$, for some $\theta \geq 2$}
\newcommand{\rwassumptions}{Consider the random walk $\{X_n\}_{n \geq 0}$ driven by $\syst$.}

\newcommand{\volsum}[1]{\sum_{d(x,y)^\theta \leq n \leq 2D^\theta}\frac{#1}{V(x,n^{1/\theta})}}
\newcommand{\volsumo}[2]{\sum_{0 \leq n \leq 2D^\theta}\frac{#1}{V(#2,n^{1/\theta})}}
\renewcommand{\a}{a_{N-2}}
\newcommand{\ai}{a_{N-1}}
\newcommand{\aii}{a_N}
\newcommand{\aint}[1]{A_{[0;#1]}}

\DeclareMathOperator{\El}{\ref{eq:E}}
\DeclareMathOperator{\VD}{\ref{eq:VD}}
\DeclareMathOperator{\PI}{\ref{eq:PI}}
\DeclareMathOperator{\CS}{\hyperref[def:CS]{CS}}

\DeclareMathOperator*{\argmin}{argmin}

\numberwithin{equation}{section}

\title{Expected hitting time estimates on finite graphs}

\author{Laurent Saloff-Coste}
\address{Department of Mathematics, Cornell University, USA.}
\email{lsc@math.cornell.edu}
\thanks{L.~Saloff-Coste's research was partially supported by NSF grant DMS-2054593.}

\author{Yuwen Wang}
\address{Institut für Mathematik,  Universität Innsbruck, Austria.}
\email{ywang@math.cornell.edu}
\thanks{Y.~Wang's research is supported by the Austrian Science Fund (FWF) project (P 34129), and partially supported by the Universität Innsbruck Early Stage Funding Program.}


\begin{document}


\begin{abstract}
  The expected hitting time from vertex $a$ to vertex $b$, $H(a,b)$, is the expected value of the time it takes a random walk starting at $a$ to reach $b$. 
  In this paper, we give estimates for $H(a,b)$ when the distance between $a$ and $b$ is comparable to the diameter of the graph, and the graph satisfies a Harnack condition.
  We show that, in such cases, $H(a,b)$ can be estimated in terms of the volumes of balls around $b$. Using our results, we estimate  $H(a,b)$ on various graphs, such as rectangular tori, some convex traces in $\Z^d$, and fractal graphs. Our proofs use heat kernel estimates. 
\end{abstract}

\maketitle


\section{Introduction}
Given a Markov chain on a graph (for us here, a finite graph), there are many reasons to be interested in the random variables  $\tau_{b}$, the time it takes for the Markov chain to make its first visit at the vertex $b$. An excellent introduction is in \cite{levinMarkovChainsMixing2017}. See also \cite{aldous-fill-2014}. In the case of  simple random walk on the $d$-torus $(\mathbb Z/N\mathbb Z)^d$, when $d$ is fixed and $N$ is a varying parameter, and for vertices $a,b$ at distance of order $N$ of each other, $$H(a,b)=\E_a[\tau_{b}]\asymp\left\{\begin{array}{cl} N^2& \mbox{ if } d=1,\\N^2\log N& \mbox{ if } d=2,\\ N^{d} &\mbox{ if  }d>2.\end{array}\right.$$
The goal of this work is try to explain these behaviors in geometric terms so that they can be extended beyond very specific examples such as the $(\mathbb Z/N\mathbb Z)^d$.  For this, we use heat kernel techniques. Comparing to  \cite[Proposition 10.21]{levinMarkovChainsMixing2017}  which gives the behavior of $\E_a[\tau_{b}]$ on  $(\mathbb Z/N\mathbb Z)^d$ in terms of the distance between $a$ and $b$ even when $a,b$ are relatively close to each other, 
our results are (mostly) limited to the case when the distance between vertices $a$ and $b$ is of order the diameter of the graph. However, we also provide a simple  condition that imply  that $\E_a[\tau_b]\asymp \E_a[\tau_c]$
for any vertices $b,c\neq a$, as it happens for instance on $(\mathbb Z/N\mathbb Z)^d$ when $d\ge 3$. From the point of view of heat kernel estimates, treating the small distance case in the most sensitive cases would require gradient (i.e., difference) estimates which are much more difficult to obtain than the heat kernel estimates themselves.

Our focus is to understand how the torus result cited above generalizes to a variety of other examples of a similar type, from rectangular tori $\prod_1^N(\mathbb Z/a_i\mathbb Z)$, $a_1\le \dots\le a_N$, to Cayley graphs of finite nilpotent groups such as the group of $3$ by $3$ upper-triangular matrices with entries mod $N$ and entries equal to $1$ on the diagonal, to lattice traces on simple convex subsets of $\mathbb R^2 $ or $\mathbb R^3$, and classes of finite fractal graph  such as the $N$-iteration in the construction of the Sierpinski gasket. See the examples in Section \ref{sect:exs}. The common thread amongst all these examples is that they are all \emph{Harnack graphs} which means that they are all amenable to sharp two-sided \emph{heat kernel bounds}. 
Here, the \emph{heat kernel} refers to the discrete time iterated kernel of the natural random walk on the graph in question, and  \emph{heat kernel bounds} refers to estimates based on basic geometric quantities such graph distance, $d$,  and volume of balls with respect to the underlying reversible probability measure $\pi$, $V(x,r)=\pi(\{y:d(x,y)\le r))$. The definition of \emph{Harnack graph} involves an important parameter, $\theta\ge 2$, and can be given in several equivalent forms. Here, we use a characterization based on four conditions, \assumptions{}. Our main result can be stated as follows.
\begin{thm} \label{thm-0}
    Suppose that the Markov kernel $\syst$ (on a finite graph $\Gamma = (\V, E)$) satisfies \assumptions{}.
Let $x,y \in \V$ such that $d(x,y) \geq c_0 D$ for some $c_0 > 0$.
Then, there exists constants $C_1, C_2 > 0$ such that
\begin{equation*}
\volsumo{C_1}{y} \leq H(x,y) \leq \volsumo{C_2}{y}.
\end{equation*}\end{thm}

The structure of the article is as follows. Section \ref{sec-back} sets notation and review key identity relying hitting times expectation to the iterated kernel of the chain. It also review deep iterated kernel estimates related to the notion of Harnack graph. Section \ref{sec-Green} derives the key estimates for what we call the "Green function." Section \ref{sec-Hit} applies the results of the previous sections to obtain the main result of this paper, Theorem \ref{thm-0} (also Theorem \ref{thm:main}). The last section describes various examples.

\section{Background}\label{sec-back}

\subsection{Random walks preliminaries}

Let $\Gamma = (\V,E)$ be a finite graph where $\V$ is the vertex set and $E\subset \V \times \V$ is the symmetric edge set. We write $x\sim y$ to signify that $(x,y)\in E$. We assume the graph is connected in the sense that one can join any two vertices by a path crossing edges.
Our random process of interest is an aperiodic, irreducible, and reversible Markov chain $K: \V \times \V \to \R$ with reversible probability measure $\pi: \V\to [0,1]$ (we use the same notation for the probability measure and its probability mass function).
To relate $K$ to the graph structure, we specify that $K(x,y)>0$ if $x\sim y$ and $K(x,y) = 0$ for all $x \neq y$ such that $x \not \sim y$. The iterated kernel is defined inductively by $K^n(x,y)=\sum_z K(x,z)K^{n-1}(z,y).$ Reversibility means that $\pi(x)K(x,y)=\pi(y)K(y,x)$ and this implies that $\pi$ is an invariant measure for $K$, $\sum_x \pi(x)K(x,y)=\pi(y)$. We say that $K$ is lazy if $K(x,x)\ge 1/2$ for all $x\in\V$. When $K$ is lazy, it is obviously aperiodic.
Let $d(x,y)$ to be the graph distance on $\Gamma$, i.e. the minimal number of edges one must cross going from $x$ to $y$. 
The \emph{diameter} of the graph, 
\[D = \max_{x,y \in \V} d(x,y), \] 
is  one of the key geometric parameter we will use.
Note that, for all $x \neq y$ and $0 \leq n < d(x,y)$, we have $K^n(x,y) = 0$. 

For the remainder of this article, we will refer to these objects  as the \emph{Markov kernel $(\Gamma, K, \pi)$}.
The \emph{normalized kernel} is 
\begin{equation*}
k^n(x,y) = \frac{K^n(x,y)}{\pi(y)}, \quad \text{for } n \geq 0.
\end{equation*}

Each Markov kernel $\syst$ defines a \emph{random walk $\{X_n\}_{n \geq 0}$ driven by $K$}, where $X_0$ has an initial distribution $\nu$ on $\V$ and for all $n \geq 1$ and $y \in \V$,
\[
\P_\nu(X_{n+1} = y) = \sum_{x \in \V} \nu(x) K^n(x,y).
\]
For any $x \in \V$, it is convenient to define the operators 
\[
  \P_x ( \cpdot ) = \P_x ( \cpdot | X_0 = x) \qquad \text{and} \qquad \E_x [ \cpdot ] = \E_x [ \cpdot | X_0 = x].
\]
The \emph{hitting time of $a \in \V$} is
$
  \tau_a = \min\{n\geq 0: X_n =a\}.
$
For $a, b \in \V$, the \emph{expected hitting time of $b$ starting at $a$} is
\begin{equation}
  H(a,b) = \E_a[\tau_b]. 
   \label{eq:exphittime}
\end{equation}
In general, this is not a symmetric function of $a,b$.
The \emph{exit time of a set $B \subseteq \V$} is
\begin{equation}
\Tau_B = \inf \{ t :X_n \not \in  B \}.
\label{eq:exittime}
\end{equation}
For $x \in \V$ and $r \geq 0$, we define the (closed) balls of radius $r$ centered at $x$ as
\begin{equation*}
B(x,r) = \{y \in \V: d(x,y) \leq r \},
\end{equation*}
and its volume as $V(x,r) = \pi(B(x,r)) = \sum_{y \in B(x,r)} \pi(y)$.

\subsection{Discrete Laplacian, Green's function, and expected hitting time}

\begin{defn}
Define the \emph{identity operator}, $I: \V \times \V$, $I(x,y) = \mathbf 1_{\{x\}}(y)$. 
The \emph{random walk Laplacian} is $\Delta = I - K$, and the \emph{(normalized) discrete Green's function} is the function $\G: \V \times \V \to \R$, where
\begin{equation}
\G (x,y) = \sum_{n = 0}^\infty (k^n(x,y) - 1).
\label{eq:green}
\end{equation}
For convenience, we will refer to $\Delta$ as \emph{the Laplacian} 
and $\G$ as \emph{the Green function}.
\end{defn}

At this point, it will useful to establish some basic facts about these operators.
Set 
\begin{equation*}
  G(x,y) = \pi(y) \G(x,y) = \sum_{n=0}^\infty (K^n(x,y)-\pi(y)).
\end{equation*}
First, note that  $\sum_y G(x,y)=\sum_y\G(x,y)\pi(y)=0$. 
Second, $\G$ is symmetric by reversibility, but the same does not apply to $G$.
Third, when viewed as operators (or matrices) acting on functions satisfying $\pi(f)=0$, the Laplacian, $\Delta$, and the Green function,  $G$, are inverse of each other. More specifically, for functions $f$ such that $\pi(f) = 0$, $[G \Delta] f = f$ as seen from the following lemma. 

\begin{lem}\label{lemGD}
  \begin{equation*}
    [G \Delta] (x,y) = I(x,y) - \pi(y)
  \end{equation*}
  \label{lem:GDinv}
\end{lem}

\begin{proof}
  Compute
    \begin{align*}
    [G K] (x,y) &= \sum_{z \in \V} \sum_{n=0}^\infty(K^n(x,z) - \pi(z)) K(z,y)
    = \sum_{n=0}^\infty (K^{n+1}(x,y) - \pi(y))
    = G(x,y) - (I(x,y) - \pi(y)).
  \end{align*}
  Then by the definition of the discrete Laplacian,
  \begin{align*}
    [G\Delta] (x,y) &= G(x,y) - [G K] (x,y)
    = G(x,y) - G(x,y) + I(x,y) - \pi(y)
    = I(x,y) - \pi(y). \qedhere
  \end{align*}
\end{proof}

\begin{lem}
  Let $H: \V \times \V \to \R$ be any kernel  such that $H(x,x) = 0$ for all $x \in \V$.
  Then, we have
  \begin{equation*}
    [G \Delta H] (x,x) = - \sum_{y \in \V} \pi(y)H(y,x).
  \end{equation*}
  \label{lem:diagonal}
\end{lem}

\begin{proof}
  Using Lemma \ref{lem:GDinv}, for all $x,y \in \V$, we have
    \begin{align*}
    H(x,y)  &= \sum_{z \in \V} ([G \Delta] (x,z) + \pi(z)) H(z,y)
    = [G \Delta H] (x,y) + \sum_{z \in \V} \pi(z) H(z,y).
    \end{align*}
  By the assumption that $H (x,x) =0$, we have
    \[
    [G\Delta H](x,x) = - \sum_{z\in \V} \pi(z) H(z,x). \qedhere
    \]
\end{proof}

\begin{prop}
[{\cite[Chapter 2, Lemma 12]{aldous-fill-2014}}]  Let $H(\cdot,\cdot)$ be the expected hitting time function defined in (\ref{eq:exphittime}).
  \rwassumptions 
  Then for all $x, y \in \V$,
  \begin{equation*}
    H(x,y) = \G (y,y) - \G(x,y).
  \end{equation*}
  \label{prop:hittinggreen}
\end{prop}

\begin{proof}
  First, we want to compute $KH(a,b)$ for all $a, b \in \V$. On the diagonal, we have
  \begin{align*}
    [KH](a,a) = \sum_{z \in \V} K(a,z) H(z,a) = \E_a[\tau^+_a] = \frac{1}{\pi(a)},
  \end{align*}
  where $\tau^+_a = \min\{t \geq 1 : X_t = a\}$.
  The last equality follows from $K$ being irreducible, see \cite[Theorem 5.5.11]{durrettProbabilityTheoryExamples2010}.
  When $a \neq b$, we have $X_0 = a$ and $X_1$ is distributed according to $\charf_aK$, and thus, the following  relation holds:
  \begin{align*}
  H(a,b) &= \E[\tau_b | X_0 = a, X_1 = b] + \sum_{z: z \neq b} \E[\tau_b | X_0 = a, X_1 = z] \\
  &= K(a,b) + \sum_{z: z \neq b} K(a,z) (1 + \E[\tau_b| X_0 = z]) \\
  &= \sum_{z \in \V}K(a,z) + \sum_{z: z\neq b} K(a,z) H(z,b) \\
  &= 1 + [KH](a,b).
  \end{align*}
  Then for all $x,y \in \V$,
  \begin{equation}
    [\Delta H] (x,y) = \begin{cases}
  1- \frac{1}{\pi(x)} &\text{if } x = y \\
  1 & \text{if } x \neq y.
  \end{cases}
  \label{eq:DH}
  \end{equation}
  Applying $G$ to both sides and using $\sum_yG(x,y)=0$, we have
  \begin{align*}
    [G \Delta H] (x,x) &= \sum_{y \in \V} G(x,y) [\Delta H](y,x) = - \frac{1}{\pi(x)} G(x,x).  \end{align*}
  Note that the hitting time (\ref{eq:exphittime}) from a point to itself is identically zero. 
  Combining this with Lemma \ref{lem:diagonal}, we have
  \begin{equation*}
    \G(x,x) = \frac{G(x,x)}{\pi(x)} = \sum_{y \in \V}  \pi(y)H(y,x).
  \end{equation*}
  Then for all $x,y$, we can apply $G$ once again to (\ref{eq:DH}) to get
  \begin{equation*}
    [G \Delta H](x,y) = \sum_{z \in \V} G(x,z) [\Delta H](z,y) = - \frac{G(x,y)}{\pi(y)} .  \end{equation*}
  Combining this with Lemma \ref{lemGD}, we get
  \begin{equation*}
    H(x,y) = [G \Delta H](x,y) + \sum_{z \in \V} \pi(z) H(z,y) = \frac{G(y,y)}{\pi(x)} - \frac{G(x,y)}{\pi(y)} = \G(y,y) - \G(x,y). \qedhere
  \end{equation*}
  
\end{proof}
We note that these properties are all essentially well known although they are often presented in slightly different ways.

\subsection{Remarks on periodicity and laziness}
Although we have assumed aperiodicity in addition to irreducibility, the aperiodicity assumption is not essential for our purpose. Indeed, the definition of the Green function, $\G(x,y)=\sum_{n=0}^\infty (k^n(x,y)-1)$ makes sense for irreducible periodic chain as well as long as on understand it in the form
$$\G(x,y)=\sum_{n=0}^\infty (k^{2n}(x,y)+k^{2n+1}(x,y)-2).$$
To see that this series converges, we use the spectral decomposition of  the reversible irreducible  kernel $K$ with eigenvalues $\beta_i\in [-1,1]$, $0\le i\le |\V|-1$, arranged in non-increasing order, and associated normalized real eigenfunctions $\phi_i$. Because $K$ is irreducible, $\beta_0=1$ and $\beta_2<1$. The chain is periodic if and only if $\beta_{|\V|-1}=-1$. This gives
$$k^{2n}(x,y)+k^{2n+1}(x,y)-2=\sum_{i=1}^{|\V|-1}(1+\beta_i)\beta_i^{2n}\phi_i(x)\phi_i(y).$$
This term decays exponentially fast because all eigenvalues equal to $-1$ drop from this sum thanks to the factor $(1+\beta_i)$.  

Define$$I_\pm(x,y)=\left\{\begin{array}{l}+1 \mbox{ if $x$ and $y$ are in the same class},\\-1
\mbox{ if $x$ and $y$ are not in the same class}.\end{array}\right.$$
The statement of Lemma \ref{lemGD} need to be adjusted to
\begin{eqnarray}GK(x,y)&=&I(x,y)-\lim_{n\to \infty}K^{2n+2}(x,y)\nonumber 
\\
&=&I(x,y)- \left\{\begin{array}{l}\pi(y) \mbox{ if $K$ is aperiodic},\\2\pi(y) \mbox{ if $K$ is periodic and $x,y$ are in the same class},\\
 0\mbox{ if $K$ is periodic and $x,y$ are not in the same class}.\end{array}\right. \label{lemGD*}\end{eqnarray}
This gives the correct result whether or not $K$ is aperiodic. In the periodic case, letting $\mathfrak C(x)$ denote the class of $x$, the identity in Lemma (\ref{lem:diagonal}) for an arbitrary kernel $H$ satisfying $H(x,x)=0$ becomes 
\begin{equation}
[G\Delta H](x,x)=- 2\sum_{z\in \mathfrak C(x)}\pi(z)H(z,x). \label{lem:diagonal*}
\end{equation}
We need to observe that for $y\not\in\mathfrak C(x)$, $H(y,x)=1+\sum_{z\in \mathfrak C(x)}K(y,z)H(z,x)$. 
Summing over $y\not\in \mathfrak C(x)$ and remembering that $\pi(\mathfrak C(x))=1/2$, we obtain
$$\sum_{y\not\in\mathfrak C(x)}\pi(y)H(y,x)= \frac{1}{2} +\sum_{z\in\mathfrak C(x)}\pi(z)H(z,x).$$

From there, one checks that the identity of Proposition \ref{prop:hittinggreen} concerning the expected hitting time $H$ and stating that
\begin{equation}
 H(x,y)=\left\{\begin{array}{ll}\G(y,y)-\G(x,y) 
 &\mbox{ if  $x,y$ are in the same class},\\
 1+\G(y,y)-\G(x,y) 
 &\mbox{ if  $x,y$ are not in the same class}, \end{array}\right.
\label{eq:periodic-hittime}
\end{equation}
continues to hold with essentially the same proof, adjusting for (\ref{lemGD*})-(\ref{lem:diagonal*}).

When computing or estimating the expected hitting times $H(a,b)$ for an aperiodic reversible chain $K$, there is no loss of generality in assuming that the chain is lazy. Indeed,  fix $\epsilon\in (0,1)$ and define
$K_\epsilon=\epsilon I +(1-\epsilon)K$.  Using the spectral decomposition again,
$$\G_\epsilon(x,y)=\sum_{n=0}^{+\infty} \sum_{i=1}^{|\V|-1}\beta_i(\epsilon)^{n}\phi_i(x)\phi_i(y)=\sum_{i=1}^{|\V|-1}\frac{1}{1-\beta_i(\epsilon)}\phi_i(x)\phi_i(y).$$
But $\beta_i(\epsilon)=\epsilon+ (1-\epsilon)\beta_i$ so that  $1-\beta_i(\epsilon)=(1-\epsilon)(1-\beta_i)$ and 
$$\G_\epsilon(x,y)=(1-\epsilon)^{-1}\G(x,y).$$
It follows that $H_\epsilon(x,y)=(1-\epsilon)^{-1}H(x,y)$ (the $\epsilon$-lazy chain behaves as if slowdown by a factor of $(1-\epsilon)^{-1}$.  The periodic case is a bit more subtle.

Recall that for a reversible irreducible periodic chain, their is two periodic class $A,A^c$ with $\pi(A)=\pi(A^c)=1/2$, and $-1$ is an eigenvalue of multiplicity $1$ with eigenfunction $\mathbf 1_A-\mathbf 1_{A^c}$.  Of course $K^{2n+1}(a,b)=0$ if $a,b$ are in the same class and $K^{2n}(a,b)=0$ if $a,b$ are not in the same class.  In this case, we find that
$$\G_\epsilon(x,y)= (1-\epsilon)^{-1}\left(\G(x,y)+ \frac{1}{2}(\mathbf 1_A(x)-\mathbf 1_{A^c}  (x))(\mathbf 1_A(y)-\mathbf 1_{A^c}(y)\right)=(1-\epsilon)^{-1}\left(\G(x,y)+ \frac{I_{\pm}(x,y)}{2}\right)$$ 
For the expected hitting time $H_\epsilon(a,b)$ with (\ref{eq:periodic-hittime}), this gives
$$H_\epsilon (a,b)= (1-\epsilon)^{-1}H(a,b).$$
Alternatively, this last result can be seen directly has follows. Consider the product probability space of $\V^\infty$ equipped the probability $\P_a$ induced by the Markov kernel $K$ (we can fix the starting point $a$) and $\{0,1\}^\infty$ with the product of Bernoulli measures with parameter $1-\epsilon$.
For $\omega=(\theta,\sigma)\in  \V^\infty\times \{0,1\}^\infty$, set $X_k(\omega)=\theta_k$ for all $k\geq 0$. 
$N_k (\sigma) = \argmin_{j} 
\{ \sum_{i=0}^j \sigma_i = k  \} = n.$

Under $\P_a$, the random variable $X_k$ is distributed according to $K^k(a,\cdot)$
and $(X_k)_0^\infty$ is a realization of the $K$-Markov chain. 
Similarly,
$Y_n$ is distributed according to $K_\epsilon^n(a,\cdot)$ and $(Y_n)_0^\infty$ is a realization of the
$K_\epsilon$-Markov chain.  
Consider 
$\tau_b(\omega)=\inf\{k: \theta_k=b\}$ and $\tau^*_b(\omega)=\inf_k\{N_k(\sigma):\theta_{N_k(\sigma)}=b\}$.
By construction, on $\{\tau_b=k\}$, $\tau_b^*=N_k$ and (using the known first moment of a negative-binomial)
$$
H_\epsilon(a,b)
=\E_a(\tau_b^*)
=\sum_{k\ge 0} \sum_{n\ge k} n \P(N_k=n) \P_a(\tau_b=k)
=(1-\epsilon)^{-1}\sum_k k \P_a(\tau_b=k)
=(1-\epsilon)^{-1}H(a,b).
$$
This shows that, as far as  $H(a,b)$ is concerned, we may just as well study the lazy version of the chain of interest. 

\subsection{Heat kernel estimates}
In this section, we recall criteria for $(\Gamma, K, \pi)$ that imply the existence of good estimates for $k^n(x,y)$, in the sense of Theorem 
\ref{thm:ds-gaussian}. 
We refer the reader to \cite{barlowRandomWalksHeat2017} for detailed exposition on the topic.
The majority of our results assume that $\syst$ satisfies the conditions
$(\El), (\PI(\theta)), (\VD),$ and $(\CS(\theta))$ for some $\theta \geq 2$. These (well-known) conditions are described below. 
The (sometimes unspecified) constants in our results will depend on $\theta$ and the constants appearing in all these conditions, but they are independent of other all other parameters such as the size of $\V$ and its diameter.  

\begin{defn}
  The Markov kernel $\syst$ satisfies \emph{ellipticity} (E) if there exists a constant $p_0$ such that for all $x \sim y$,
  \begin{equation}
    \tag{E}
    K(x,y) \geq p_0.
    \label{eq:E}
  \end{equation}
\end{defn}

\begin{defn}
  The Markov kernel $(\Gamma, K, \pi)$ satisfies \emph{volume doubling} (VD) if
  there exists a constant $C_D > 0$ such that for all $x \in \V$ and $r > 0$,
  \begin{equation}
    V(x, 2r) \leq C_D V(x, r).
    \tag{VD}
    \label{eq:VD}
  \end{equation}
\end{defn}

\begin{defn}
  The Markov kernel $(\Gamma, K, \pi)$ satisfies \emph{the Poincar\'e inequality} (PI$(\theta)$) if there is a constant $C_P>0$ such that, for all $x \in \V$ and all $r>0$,
  the following statement is true:
  \begin{equation}
  \forall f, \quad \sum_{z \in B}\left|f(z)-f_B\right|^2 \pi(z)
  \leq C_P r^\theta \sum_{\xi, \zeta \in B, \{\xi, \zeta\} \in E}|f(\xi)-f(\zeta)|^2 K(\xi,\zeta) \pi(\xi) \text {, }
  \tag{PI}
  \label{eq:PI}
  \end{equation}
  where $B = B(x,r)$ and $f_B=\pi(B)^{-1} \sum_B f \pi$.
\end{defn}

\begin{defn}
  \label{def:CS}
  Fix $\theta \in[2, \infty)$. The Markov kernel $(\Gamma, K, \pi)$ satisfies \emph{the cut-off function existence property} $(\operatorname{CS}(\theta))$ if there are constants $C_1, C_2, C_3$ and $\epsilon>0$ such that, for any $x \in \V$ and $r>0$, there exists a cut-off function $\sigma=\sigma_{x, r}$ satisfying the following properties:
  \begin{enumerate}[label=(\alph*)]
    \item $\sigma \geq 1$ on $B(x, r / 2)$
    \item $\sigma \equiv 0$ on $\V \backslash B(x, r)$
    \item For all $y, z \in \V,|\sigma(z)-\sigma(y)| \leq C_1(d(z, y) / r)^\epsilon$
    \item For any $s \in(0, r]$ and any function $f$ on $B(x, 2 r)$,
    $$
    \begin{aligned}
      & \sum_{z \in B(x, s)}|f(z)|^2 \sum_{y:\{z, y\} \in E}|\sigma(z)-\sigma(y)|^2 K(z,y) \pi(z) \\
      & \quad \leq C_2(s / r)^{2 \epsilon}\left\{\sum_{\substack{z, y \in B(x, 2 s) \\
      \{z,y\} \in E}}|f(z)-f(y)|^2 K(z,y) \pi (z) +s^{-\theta} \sum_{z \in B(x, 2 s)}|f(z)|^2 \pi(z)\right\} .
    \end{aligned}
    $$
  \end{enumerate}
 \end{defn}
 Note that when $\theta=2$, $(\CS(2))$ is always automatically satisfied.  For other values of $\theta$, it is typically extremely difficult to verify this condition based on a description of the graph.

\begin{thm}\cite{barlowStabilityParabolicHarnack2004}
  Suppose that the Markov kernel $\syst$ satisfies \assumptions{}.
  Then there exists constants $c_1, c_2, C_1, C_2>0$ such that for all $x,y \in \V$
  $$
  k^n(x, y) \leq \frac{C_1}{V\left(x, n^{1 / \theta}\right)} \exp \left(-C_2 \left(\frac{d(x, y)^\theta}{n}\right)^{1 /(\theta-1)}\right),
  $$
  and, for all $d(x,y) \leq n$,
  $$
  \frac{c_1}{V\left(x, n^{1 / \theta}\right)} \exp \left(-c_2 \left(\frac{d(x, y)^\theta}{n}\right)^{1 /(\theta-1)}\right) \leq
  k^n(x, y) + k^{n+1}(x,y).
  $$
  \label{thm:ds-gaussian}
\end{thm}

The following lemma gives the exponential decay of $|k^n - 1|$ when $n$ is of order greater than $D^\theta$. 
The proof uses a standard interpolation argument, previously used in the proofs of
\cite[Lemma 1.1]{diaconisNashInequalitiesFinite1996} and
\cite[Theorem 6.4]{diaconisAnalyticgeometricMethodsFinite2020}.

\begin{lem}
  Suppose that the Markov kernel $\syst$ satisfies \assumptions{} and is lazy.
  Then, there exists $C_2$ and, for any $c_0>0$, there exists $C_1$ such that,
  for all $n \geq c_0 D^\theta$  and $x,y \in \V$,
  \begin{equation*}
    |k^n(x,y) - 1| \leq C_1 \exp(- C_2 n/ D^\theta),
  \end{equation*}
  ($C_1$ depends on $c_0$, but not $C_2$). 
  \label{lem:spectralgap}
\end{lem}

\begin{proof}
  This proof relies on the language of operator norms. 
  For a more thorough explanation about the analysis of reversible finite Markov chains, we refer the reader to \cite{saloff-costeLecturesFiniteMarkov1997}.
  Recall that the space $\ell^p(\pi)$ is the set of functions from $\V$ to $\R$ under the norm
  $||f||_p = \left(\sum_{x \in \V} |f(x)|^p \pi(x) \right)^{1/p}$ if $p \geq 1$ and
  $||f||_\infty = \sup_{x\in \V} |f(x)|.$
  Given $p,q \in [1,\infty]$ and $K: \ell^p(\pi) \to \ell^q(\pi)$, define
  $$||K||_{p \to q} = \sup_{f \in \ell^p(\pi)} \left \{ \frac {||Kf||_q}{||f||_p} \right \}.$$
  Given this notation, we set $n_1, n_2$ such that $n = n_1 + 2 n_2$ and have
  \begin{align*}
    |k^n(x,y) - 1| &\leq ||K^n-\pi||_{1 \to \infty} \\
    &\leq ||K^{n_2}-\pi||_{1 \to 2} ||K^{n_1}-\pi||_{2 \to 2} ||K^{n_2}-\pi||_{2 \to \infty}.
  \end{align*}
  By reversibility of $K$, we know that 
  \begin{equation*}
  ||K^{n_2}-\pi||_{1 \to 2} = ||K^{n_2}-\pi||_{2 \to \infty}
  = \sqrt{\sup_{x,y \in \V} k^{2n_2}(x,y)}
  \leq \sqrt{\sup_{x \in \V} k^{2n_2}(x,x)}. 
  \end{equation*}
  By Theorem \ref{thm:ds-gaussian}, there exists $C_1, C_2 >0$ such that 
  \begin{align*}
  k^{2 n_2}(x,y) \leq \frac{C_1}{V\left(x, (2n_2)^{1 / \theta}\right)} \exp \left(-C_2 \left(\frac{d(x,y)^\theta}{2 n_2}\right)^{1 /(\theta-1)}\right).
  \end{align*}
  By choosing $n_2$ to be the largest integer less than $c_0 D^\theta$, the exponential term is bounded by a constant depending on $c_0$. 
  The volume doubling property ($\VD$) implies
  \begin{align*}
      \frac{1}{V\left(x, (2n_2)^{1 / \theta}\right)} 
      \leq \frac{V(x,D)}{V\left(x, (2n_2)^{1 / \theta}\right)} 
      \leq C_3 \left( \frac{D}{n_2^{1/\theta}} \right)^{C_4},
  \end{align*}
  for some $C_3, C_4 > 0$. 
  Thus, we have
  \begin{equation}
  ||K^{n_2}-\pi||_{1 \to 2} ||K^{n_2}-\pi||_{2 \to \infty}
  \leq C_5 \left( \frac{D}{n_2^{1/\theta}} \right)^{C_4}
  \label{eq:spectral1}.
  \end{equation}
  By ($\El$) and ($\PI(\theta)$) and the assumed laziness of $K$, there exist  $C_6,C_7 > 0$ such that 
  \[
  ||K^{n_1}-\pi||_{2 \to 2} = (1 -  C_6/D^\theta)^{n_1} \leq
  \exp( - C_7 n_1/D^\theta),
  \]
  Combined with (\ref{eq:spectral1}), there exists a constant $C_8 >0$, such that
  \[
    |k^n(x,y) - 1| \leq  C_5 \left( \frac{D}{n_2^{1/\theta}} \right)^{C_4}
    \exp( - C_7 n_1/D^\theta) \leq \exp( - C_8 n/D^\theta). 
  \]
\end{proof}

\section{Green's function estimates}
 \label{sec-Green}
\begin{prop}
  Suppose that the Markov kernel $\syst$ is lazy and satisfies \assumptions{}.
  Then there exists a constant $C > 0$ such that for all $x,y \in \V$,
  \begin{equation}
    |\G(x,y)| \leq \sum_{d(x,y)^\theta \leq n \leq 2D^\theta}
    \frac{C}{V(x,n^{1/\theta})}.
    \label{eq:greenupper}
  \end{equation}
  \label{prop:greenupper}
\end{prop}

\begin{proof}
  We use $k^n(x,y)+1$ to upper bound the summand of (\ref{eq:green}) for $0 \leq n \leq 2D^\theta$ and break up the sum of as follows:
  \begin{equation}
    |\G(x,y)| \leq 2D^\theta
    + \sum_{0 \leq n < d(x,y)^\theta} k^n(x,y)
    + \sum_{d(x,y)^\theta \leq n \leq 2D^\theta} k^n(x,y)
    + \sum_{n > D^\theta } |k^n(x,y)-1|.
    \label{eq:greenupper1}
  \end{equation}
  Now we will show that each of the four summands above can be bounded above by a constant multiple of the sum in (\ref{eq:greenupper}).
  First note that for all $D^\theta \leq n \leq 2D^\theta$, we have $V(x,n^{1/\theta}) = 1$. Thus, we have for the first summand 
  \begin{equation*}
  D^\theta \leq \volsum{1}.
  \end{equation*}  
  
  By Theorem \ref{thm:ds-gaussian}, we know that there exists $C_1, C_2 > 0$ such that $k^n(x,y) \leq u(n) v(n)$ for all $n \geq d(x,y)$, where
  \begin{equation*}
    u(t) = \frac{C_1}{V\left(x, t^{1 / \theta}\right)} \and v(t) = \exp \left(-C_2 \left(\frac{d(x, y)^\theta}{t}\right)^{1 /(\theta-1)}\right).
  \end{equation*}
  For the second summand of (\ref{eq:greenupper1}), we use Lemma \ref{lem:appen1} and the fact that $k^{n}(x,y) = 0$ for $n < d(x,y)$ to get
  \begin{align*}
  \sum_{0 \leq n < d(x,y)^\theta} k^n(x,y)
  \leq \sum_{d(x,y) \leq n < d(x,y)^\theta} u(n) v(n) \leq \frac{C_3 d(x,y)^\theta}{V(x, d(x,y))}.
  \end{align*}
  By $(\VD)$, there exists a constant $C_D$ such that
  \begin{align*}
  \frac{d(x,y)^\theta}{V(x,d(x,y))} &\leq \frac{C_D d(x,y)^\theta}{V(x,2 d(x,y))}
  \leq \sum_{d(x,y)^\theta \leq n \leq 2 d(x,y)^\theta} \frac{C_D}{V(x,n^{1/\theta})},
  \end{align*}
  as desired. 
   For the third summand of (\ref{eq:greenupper1}), we simply use that $v(t) \leq 1$ to get
  \[
  \sum_{d(x,y)^\theta \leq n \leq 2D^\theta} k^n(x,y) \leq C_1 \volsum 1.
  \]
  In the range of the fourth summand of (\ref{eq:greenupper1}), we know that $V(x,n^{1/\theta}) = 1$.
  By Lemma \ref{lem:spectralgap} there exists constants $C_5, C_6, C_7> 0$ such that
  \begin{equation}
    \sum_{n > 2D^\theta } |k^n(x,y) - 1|
    \leq C_5 \sum_{n > 2D^\theta } \exp(- C_6 n/D^\theta)
    \leq C_5 D^\theta \int_1^{\infty} \exp(- C_6 u) \; du
    \leq C_7 D^\theta.
    \label{eq:exptail}
  \end{equation}
\end{proof}

\begin{cor}
Suppose that the Markov kernel $\syst$ is lazy and satisfies \assumptions{}.
Let $x, y \in \V$ such that $d(x,y) \geq c_0 D$ for some $c_0 > 0$.
Then there exists $C_1 > 0$ (depending on $c_0$ but not $x,y$) such that
\begin{equation*}
    |\G(x,y)| \leq C_1 D^\theta.
\end{equation*}
\label{cor:greenlowerdiameter}
\end{cor}

\begin{proof}
We simply apply Proposition \ref{prop:greenupper} to get
\begin{align*}
|\G(x,y)| &\leq \sum_{d(x,y)^\theta \leq n \leq 2D^\theta} \frac{C_1}{V(x,n^{1/\theta})}
\leq \sum_{c_0^\theta D^\theta \leq n \leq 2D^\theta}\frac{C_1}{V(x,n^{1/\theta})}
\leq (2-c_0^\theta) D^\theta \frac{1}{V(x, c_0 D)}.
\end{align*}
By Lemma \ref{lem:constantvol}, we know that $V(x, c_0 D)$ is bounded below by a constant, which gives the desired result.
\end{proof}

\begin{prop}
Suppose that the Markov kernel $\syst$ is lazy and  satisfies \assumptions{}.
Then there exists $C_1, C_2 > 0$ such that for all $x \in \V$
\begin{equation*}
  \G(x,x) \geq  \sum_{0 \leq n \leq 2D^\theta} \frac{C_1}{V(x,n^{1/\theta})} - C_2 D^\theta.
\end{equation*}
\label{prop:greenlower}
\end{prop}

\begin{proof}
In the sum of (\ref{eq:green}), we use the lower bound of $-|k^n(x,x)-1|$ for $n > 8D^\theta$, and break up the sum as follows:
\begin{equation}
\G(x,x) \geq - 4D^\theta + \sum_{0 \leq n \leq 4D^\theta} k^n(x,x)
- \sum_{n > 4D^\theta} |k^n (x,x) - 1|.
\label{eq:glower}
\end{equation}
For the second summand of (\ref{eq:glower}), we use Theorem \ref{thm:ds-gaussian} with $d(x,x) = 0$ to get
\begin{align*}
  \sum_{0 \leq n \leq 4D^\theta} k^n(x,x)
  \geq \frac12 \sum_{0 \leq n \leq 2D^\theta} (k^n(x,x) + k^{n+1}(x,x))
  \geq \frac12 \volsumo{C_1}{x}.
\end{align*}
For the third summand of (\ref{eq:glower}), we use the same technique that we used for the proof of Proposition \ref{prop:greenupper} in (\ref{eq:exptail}) to get
\begin{equation*}
  \sum_{n > 4D^\theta} |k^n (x,x) - 1| \leq C_2 D^\theta,
\end{equation*}
for some $C_2 > 0$.
Putting these bounds together, we get the desired inequality.
\end{proof}

\section{Exit time estimates} \label{sec-Hit}

\begin{lem}
    Let $o \in \V$ and $\{X_t\}_{t \geq 0}$ be the random walk driven by a Markov kernel $\syst$ with $X_0 = o$.
    For all $R > 0$ and $t \geq 0$,
    \begin{equation}
    \P_o \left( \sup_{0 < s \leq t} d(o, X_s) > R \right)
    \leq 2 \sup_{0 < s \leq t} \left( \sup_{x \in B(o,R+1)} \P_x ( d(x,X_s) > R/2) \right)
    \label{eq:exit1}
    \end{equation}
    \label{lem:exit1}
\end{lem}


\begin{proof}
  Consider a random walk $\{X_n\}_{n \geq 0}$ starting at $o \in \V$, such that $\sup_{0 < s \leq t} d(o, X_s) > R$ for some $R, t \geq 0$. 
  Recall from (\ref{eq:exittime}) that the exit time of $B \subset \V$ is $\Tau_B = \inf\{t:X_n \not \in B \}.$
  Define the events 
  \begin{align*}
    E_1 &= \{ T_{B(o,R)} \leq t \} = \{\exists s \leq t: d(o,X_s) > R \} = \left\{ \sup_{0 \leq s \leq t} d(o,X_s) > R \right\} \\
    E_2 &= \{ d(o,X_t) \leq R/2 \}.
  \end{align*}
  We will prove the desired inequality (\ref{eq:exit1}) via 
  \begin{equation}
    \P_o (E_1) \leq \P_o(E_1 \cap E_2) + \P_o(E_2^c).
    \label{eq:exit2}
  \end{equation}
  The first summand can be written as
  \begin{equation}
    \P_o \left(E_1 \cap E_2 \right) = \sum_{k = 0}^t \P_o \left(T_{B(o,R)} = k \right) 
    \P_o \left(E_1 \cap E_2 | T_{B(o,R)} = k \right). 
    \label{eq:exit3}
  \end{equation}
  \begin{figure}
  \includegraphics[width=8cm]{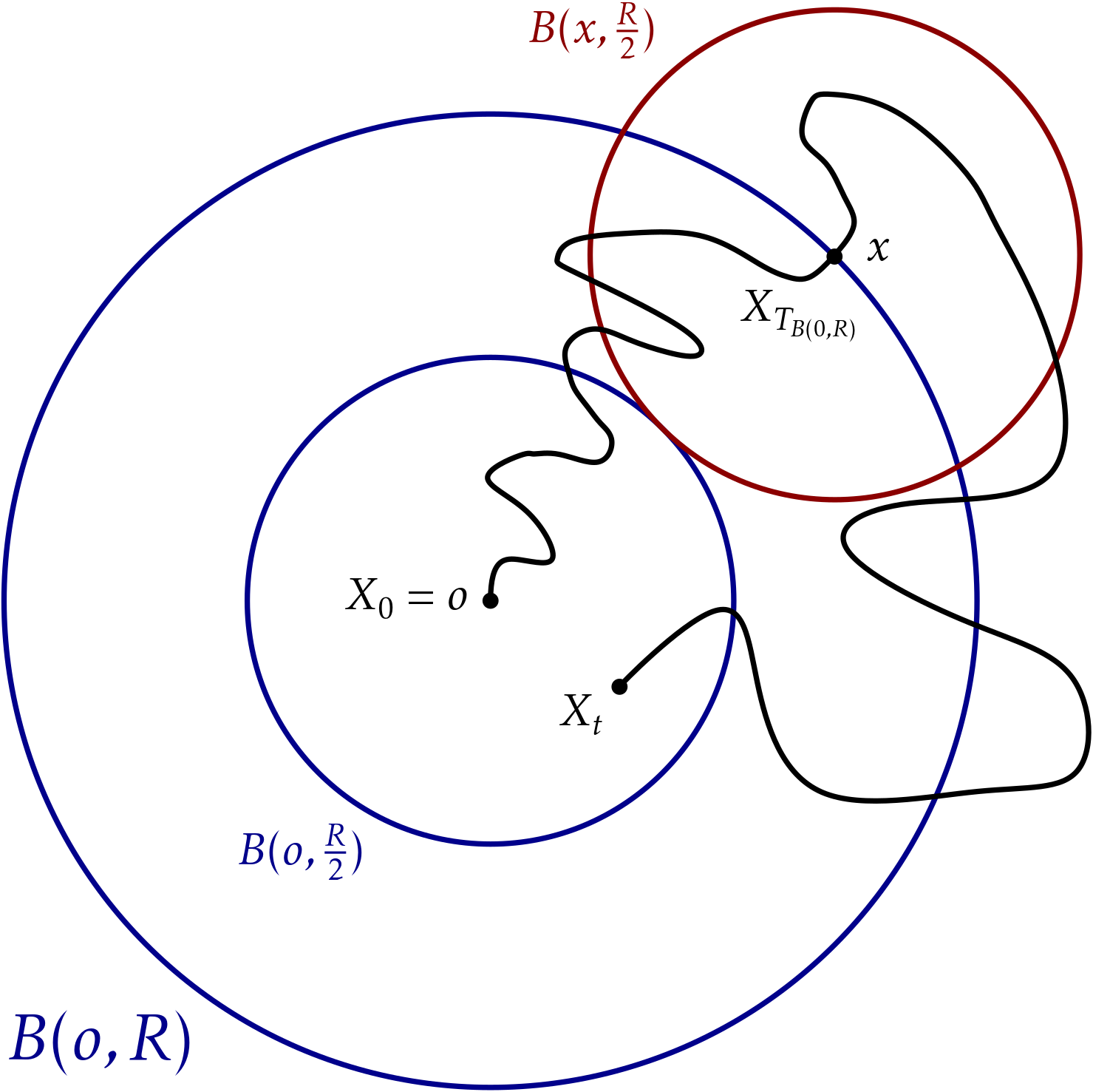}
  \caption{In the event $E_1 \cap E_2$, the walk exits the ball $B(o,R)$ for the first time at say point $x$. Then if the walk returns to the ball $B(o,R/2)$ at time $t$, it must have also exited the ball $B(x, R/2)$ by then.}
  \label{fig:exit}
  \end{figure}
  For a fixed $k \in [0, t]$, consider the event $E_1 \cap E_1$ given $T_{B(o,R)} = k$. 
  This means that the walk starts at $o$, reaches distance $R+1$ from $o$ at time $k$, and returns to the ball of $B(o,R/2)$ by time $t$, see Figure \ref{fig:exit}. By the triangle inequality, we have 
  \begin{equation*}
    R+1 = d(o,X_t) \leq d(o, X_k) + d(X_k, X_t) \leq \frac R2 + d(X_t,X_k).
  \end{equation*}
  Thus, we have $d(X_t, X_k) > R/2$, and
  \begin{align*}
    \P_o\left(E_1 \cap E_2 | T_{B(o,R)} = k \right) &\leq
    \sup_{x \in B(o,R+1)} \P_o \left( \left. d(X_t, x) > \frac R2 \; \right| \; T_{B(o,R)} = k, X_k = x \right).
  \end{align*}
  By the strong Markov property, for any $x \in B(o,R+1)$, we have
  \begin{align*}
     \P_o \left( \left. d(X_t, x) > \frac R2 \; \right| \; T_{B(o,R)} = k, X_k = x\right)
       &=  \P \left( \left. d(X_{t-k}, x) > \frac R2 \; \right| \; X_0 = x \right)\\
         &\leq \sup_{0 \leq s \leq t} \P_x \left(  d(X_s, x) > \frac R2 \right).
  \end{align*}
  Combined with (\ref{eq:exit3}), the above implies
  \begin{equation}
    \P_o \left(E_1 \cap E_2 \right) \leq \sup_{0 \leq s \leq t} \left( \sup_{x \in B(o,R+1)} \P_x \left(  d(X_s, x) > \frac R2 \right) \right).
    \label{eq:exit4}
  \end{equation}
  It is also clear that 
  \[
    \P_o(E_2^c) = \P_o \left( d(o,X_t) > R/2 \right) \leq \sup_{0 \leq s \leq t} \left( \sup_{x \in B(o,R+1)} \P_x \left(  d(X_s, x) > \frac R2 \right) \right).
  \]
  This, with (\ref{eq:exit2}) and (\ref{eq:exit4}), gives our desired result.
\end{proof}

\begin{lem}
Suppose that the Markov kernel $\syst$ satisfies \assumptions{}.
Let $x \in \V$ and $\{X_n\}_{n \geq 0}$ be the canonical random walk driven by $K$ with $X_0 = x$.
Then, there exists a constant $C > 0$, independent of $x$ such that
\begin{equation}
    \P_x ( d(x,X_s) > R) \leq C \exp\left( -(R^\theta/s)^{\frac{1}{\theta-1}}\right),
    \label{eq:diasum1}
\end{equation}
for all $R, s \geq 0$.
\label{lem:exit2}
\end{lem}

\begin{proof}
  Note that if $R^\theta \leq s$, the right hand side of (\ref{eq:diasum1}) is greater than one by choosing $C \geq 1$, so the inequality is trivially satisfied. 
  Thus, we solely consider the case when $R^\theta \geq s$. 
  First, the probability in (\ref{eq:diasum1}) can be written in terms of the kernel and the normalize kernel as: 
  \[
    \P_x ( d(x,X_s) > R) = \sum_{y \in \V: d(x,y) > R} K^s(x,y) = \sum_{y \in \V: d(x,y) > R} k^s(x,y) \pi(y).
  \]
  By Theorem \ref{thm:ds-gaussian}, we have that there exists $C_1, C_2 > 0$ such that
  \begin{equation}
    \P_x ( d(x,X_s) > R) \leq C_1 \sum_{y \in \V: d(x,y) > R} \frac{\pi(y)}{V\left(x, s^{1 / \theta}\right)} \exp \left(-C_2 \left(\frac{d(x, y)^\theta}{s}\right)^{1 /(\theta-1)}\right).
    \label{eq:exit10}
  \end{equation}
  Now, we break up the sum on the right hand side into partitions of vertices $y$ such that $2^k R < d(x,y) \leq 2^{k+1} R$: 
  \begin{align*}
    \P_x ( d(x,X_s) > R)  
    &\leq C_1 \sum_{k=0}^\infty \frac{V(x,2^{k+1} R)}{V\left(x, s^{1 / \theta}\right)} \exp \left(-C_2 \left(\frac{(2^k R)^\theta}{s}\right)^{1 /(\theta-1)}\right) \\
    &\leq C_3 \sum_{k=0}^\infty \left( \frac{2^k R}{s^{1/\theta}} \right)^{C_4} \exp \left(-C_2 \left(\frac{2^k R}{s^{1/\theta}}\right)^{\frac{\theta}{\theta-1}}\right),
  \end{align*}
  where the last inequality is by ($\VD$). 
  Since $\frac{\theta}{\theta-1} > 1$, there exists a constant $C_5 > 0$ such that for all $x \geq 1$, 
  \[ 
    x^{C_4} \exp \left(-C_2 x^{\frac{\theta}{\theta-1}}\right) \leq 
    C_5 \exp \left(- \frac{C_2 x^{\frac{\theta}{\theta-1}} }{2} \right).
  \]
  Recall that $R^\theta > s$, so we can apply the above statement to simplify (\ref{eq:exit10}) further: 
  \begin{align*}
    \P_x ( d(x,X_s) > R) 
    &\leq C_6 \sum_{k=0}^\infty \exp \left( - C_7 \left( \frac{2^k R}{s^{1/\theta}} \right)^{\frac{\theta}{\theta-1}} \right)
    \leq C_6 \exp \left( - C_7 \left( \frac{R^\theta}{s} \right)^{\frac{1}{\theta-1}} \right)
    \left( 1 +  \sum_{k=1}^\infty \exp \left( - C_7 2^k \right) \right) \\
    &\leq C_8 \exp \left( - C_7 \left( \frac{R^\theta}{s} \right)^{\frac{1}{\theta-1}} \right).
  \end{align*}
\end{proof}

\section{Expected hitting time estimates}

We first use the exit time estimates established in the previous section to show that the expected hitting time is at least of order $D^\theta$.

\begin{lem}
Suppose that the Markov kernel $\syst$ satisfies \assumptions{}.
  Let $x,y \in \V$ such that $d(x,y) \geq c_0 D$ for some $c_0 > 0$.
  Then, there exists $C > 0$ such that
  \begin{equation*}
    H(x,y) \geq C D^\theta.
  \end{equation*}
  \label{lem:hittinglower}
\end{lem}
\begin{proof}
  Let $r = d(x,y) \geq c_0 D$.
  If the random walk starting at $x$ hits $y$, it must exit the ball of radius $r/2$. 
  Applying Markov's inequality, we have that
  \begin{equation}
    \E_x [ \tau_y] > \E_x [T_{B(x,r/2)}] > t \P_x [T_{B(x,r/2)} > t], 
    \label{eq:markov}
  \end{equation}
  for all $t > 0$. 

  In addition, by Lemma \ref{lem:exit2}, there exists a constant $C > 0$, such that
  \begin{equation*}
      \P_x ( d(x,X_s) > r/4) \leq C \exp\left( -(r^\theta/s)^{\frac{1}{\theta-1}}\right),
  \end{equation*}
  for all $x \in \V$ and $s \geq 0$.
  We can choose a constant $C_0 > 0$ such that for all $x \in \V$ and $0 \leq s \leq C_0r^\theta$, it is true that
  \begin{equation}
  \P_x ( d(x,X_s) > r/4) < 1/4.
  \label{eq:exit0}
  \end{equation}
  Applying Lemma \ref{lem:exit1} with $t = C_0r^\theta$, we get 
   \begin{align*}
    \P_x (T_{(B(x,r/2)} \leq t) &= \P_x \left( \sup_{0 < s \leq t} d(x, X_t) > r/2 \right)
    \leq 2 \sup_{0 < s \leq t} \left( \sup_{z \in B\left(x,\frac r2+1\right)} \P_z ( d(z,X_s) > r/4) \right) \\
    &\leq 2 \sup_{0 < s \leq t} \left(C \exp\left( -(r^\theta/s)^{\frac{1}{\theta-1}}\right) \right) 
\leq \frac 12  \qquad \textrm{(by (\ref{eq:exit0})).} 
  \end{align*}
  Combined with (\ref{eq:markov}), there exists $C_1 >0$ such that
  \[
    \E_x [ \tau_y] \geq \frac{C_0 r^\theta}{2} \geq C_1 D^\theta. 
  \]
\end{proof}

\begin{thm}
\label{thm:main}
Suppose that the Markov kernel $\syst$ satisfies \assumptions{}.
Let $x,y \in \V$ such that $d(x,y) \geq c_0 D$ for some $c_0 > 0$.
Then, there exists constants $C_1, C_2 > 0$ such that
\begin{equation*}
\volsumo{C_1}{y} \leq H(x,y) \leq \volsumo{C_2}{y}.
\end{equation*}
The upper bound is valid, in fact, for all $x,y \in \V$.
\end{thm}

\begin{proof}
Without loss of generality, we can assume that $K$ is lazy. For the upper bound, Proposition \ref{prop:greenupper}
and the fact that $\G(\cdot, \cdot)$ is symmetric (by reversibility) give
\begin{align*}
H(x,y) &\leq |\G(y,y)| + |\G(y,x)| 
\leq \volsumo{C_1}{y} + \sum_{d(x,y)^\theta \leq n \leq 2D^\theta} \frac{C_1}{V(y,n^{1/\theta})} 
\leq \volsumo{2C_1}{y}.
\end{align*}
This gives the desired upper bound.

For the lower bound, by Corollary \ref{cor:greenlowerdiameter}, Proposition \ref{prop:greenlower}, and Lemma \ref{lem:hittinglower}, we have
\begin{align*}
H(x,y)
&\geq \max \left\{ \volsumo{C_4}{y} - C_5 D^\theta, C_6 D^\theta \right\}
 \geq \volsumo{C_7}{y},
\end{align*}
for some $C_7 > 0$. 
\end{proof}

\section{Applications and Examples}
\label{sect:exs}
In this section, we describe some interesting applications of our main result, Theorem \ref{thm:main}, to compute $H(x,y)$ for various graphs. 
For brevity, we will write $f(x) \asymp g(x)$ for non-negative functions $f,g$, if there exists $c > 0$ such that $f(x) \leq c g(x)$ and $f(x) \geq c^{-1} g(x)$ for all $x$.

\subsection{Expected hitting times and resistances}
The connection between random walks and electric networks is well known, see \cite{doyle1984random}. Given a random walk driven by a Markov kernel $\syst$, we can define an electric resistance network on $\Gamma$ with conductance
\[
c(x,y) = K(x,y) \pi(x).
\]
Then when $\syst$ and $x,y \in \V$ satisfy the hypothesis of Theorem \ref{thm:main}, the effective resistance between $x$ and $y$ is 
\[\mathcal{R} (x \leftrightarrow y) \asymp H(x,y) + H(y,x) \asymp 
\max \left\{ \sum_{0 \leq n \leq 2D^\theta}\frac{1}{V(y,n^{1/\theta})}, 
\sum_{0 \leq n \leq 2D^\theta}\frac{1}{V(x,n^{1/\theta})}
\right\}.
\]

\begin{figure}
  \centering
       \begin{subfigure}[b]{0.36\textwidth}
         \centering
         \includegraphics[width=\textwidth]{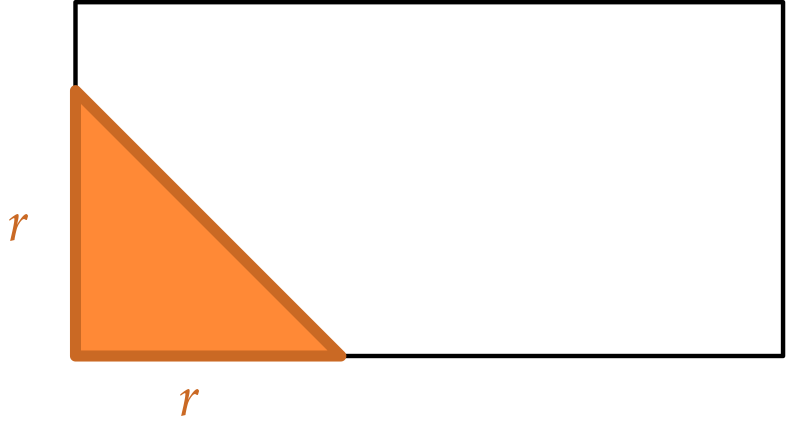}
         \caption{$r \leq b$}
     \end{subfigure}
     \hfill
     \begin{subfigure}[b]{0.36\textwidth}
         \centering
         \includegraphics[width=\textwidth]{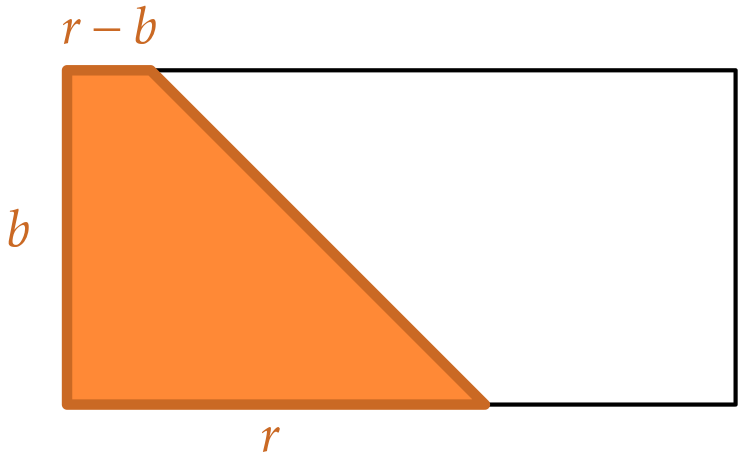}
         \caption{$r > b$}
     \end{subfigure}
  \caption{The shape of $B(o,r)$ changes at $r = b$.}
  \label{fig:rectangle}
\end{figure}

\subsection{Rectangular tori}
\label{sect:tori}
There is extensive literature on computing $H(x,y)$ for various graphs. One important early result is the following theorem by Cox: 

\begin{prop}\cite[Theorem 4]{coxCoalescingRandomWalks1989}
  Consider the simple random walk on the torus $\mathbb{Z}_n^d$. There exist constants $0<c_d \leq C_d<\infty$ such that if $x$ and $y$ are at distance order $n$, the diameter of the torus, then
  \begin{align}
  c_d n^d           &\leq H(x,y) \leq C_d n^d           &&  \text { if } d \geq 3 \nonumber \\ 
  c_2 n^2 \log (n)  &\leq H(x,y) \leq C_2 n^2 \log (n)  &&  \text { if } d=2 \nonumber \\
  c_1 n^2           &\leq H(x,y) \leq C_1 n^2           &&  \text { if } d=1.
  \label{eq:torus-hittime}
  \end{align}
\end{prop}
Early proofs of this result involve careful estimation of the spectrum of the kernel and transition probabilities. Later, \cite[Proposition 10.21]{levinMarkovChainsMixing2017} provides a proof using the relationship of $H(x,y)$and $\mathcal{R}(x \leftrightarrow y)$. 
Then the resistance is estimated from test functions constructed from P\'olya urns. 
The complexity of these proofs increases greatly even when we consider tori with different side lengths.

Consider a rectangular $2$-dimensional torus, $\Gamma = \Z_a \times \Z_b$, where $a \geq b$. 
It is known that that assumptions of Theorem \ref{thm:main} is satisfied with $\theta = 2$. 
Let $x,y \in \Gamma$ be such that $d(x,y) \asymp a$, i.e. of the order of the diameter of $\Gamma$. 
Note that around $r = b$, the volume growth of balls change from quadratic to linear growth, see Figure \ref{fig:rectangle}, and we have 
\begin{equation*}
\#B(x,r) \asymp \begin{cases}
 r^2 &\text{if } r \leq b \\
 b^2 + b(r-b) &\text{if } b < r \leq a.
\end{cases}
\end{equation*}
Applying these estimates to Theorem \ref{thm:main} with $\theta = 2$ , we get 
\begin{equation}
  H(x,y) \asymp ab \left(\log(b) + \frac 1 b (a-b)\right) \asymp \max( a b \log b, a^2).
  \label{eq:2tori}
\end{equation}
When $a = b = n$, the above formula agrees with the known estimates (\ref{eq:torus-hittime}) for the $2$-torus, $H(x,y) \asymp n^2 \log n$. Similarly, when $a = n$ and $b=1$, we have $H(x,y) \asymp n^2$, which matches the $d = 1$ case. In general, we have the following result:

\begin{prop} Fix the integer $N$ (the dimension).
  Consider the simple random walk on the rectangular torus 
  \[
  \Gamma = \Z_{a_1} \times \Z_{a_2} \times \dotsb \times \Z_{a_N}, 
  \]
  where $1 \leq a_1 \leq a_2 \leq \cdots \leq a_N$. 
  Then for $x, y  \in \V$ such that $d(x,y) \asymp D$, 
  \begin{equation}
    H(x,y) \asymp_N \max \left\{ \prod_{i=1}^N a_i, a_N a_{N-1} \log \left(\frac{a_{N-1}}{a_{N-2}}\right), a_N^2\right\}.
    \label{eq:rect}
  \end{equation}
  If $N < 3$, one simply uses $N=3$ and set appropriate $a_i$'s to $1$.
  Note that with $N=3$ and $a_1=1$, the above expression reduces to (\ref{eq:2tori}).
  \label{prop:rect}
\end{prop}

\begin{proof}
In order to apply Theorem \ref{thm:main}, we first compute the sizes of balls in $\Gamma$. 
For convenience, we set $a_0 =1$ and  introduce the notation
\[
A_{[j_1; j_2]} := \prod_{j=j_1}^{j_2} a_j.
\]
Moreover, since the torus is vertex transient, we can set $B(r) = B(x,r)$ and $V(r) = B(x,r)$ for any $x \in V$. 
Note that as $r$ grows past $a_i$, the contribution to the volume in the $i^{th}$ coordinate is fixed at $a_i$. 
Thus, for any $x \in \V$ and $0 \leq i \leq N-1$, we have 
\begin{equation}
        A_{[0;i]}  \left( \frac{r}{2} \right)^{N-i}
        \leq \# B(r) \leq  
        A_{[0;i]}  (2r)^{N-i}, \qquad \qquad \text{for } a_{i} < r \leq a_{i+1}.
        \label{eq:rect0}
\end{equation}
In particular, when $0 \leq r \leq a_{N-2}$, we know that $\# B(r) \geq 2^{-N} r^3$. 
This implies that 
\begin{equation}
A_{[0;N]} \leq \sum_{n = 0}^{a_{N-2}^2} \frac{1}{V(\sqrt{n})} 
\leq A_{[0;N]} 2^N \left( 1 + \sum_{n=1}^\infty n^{-3/2} \right). 
\label{eq:rect-a}
\end{equation}
Now we establish two useful facts. First, 
suppose that for a particular $i$, we have $a_{i+1} \leq 2 a_i$, i.e. 
$a_{i+1}/2 \leq a_i \leq a_{i+1}$. 
In this case, (\ref{eq:rect0}) become
\begin{equation}
        A_{[0;i+1]} \left(\frac 14\right) \left( \frac{r}{2} \right)^{N-(i+1)}
        \leq \# B(r) \leq  
        A_{[0;i+1]} 2 (2r)^{N-(i+1)}, \qquad \qquad \text{for } a_{i} < r \leq a_{i+1}.
        \label{eq:rect1}
\end{equation}
In other words, if $a_i$ and $a_{i+1}$ are of comparable size, the estimate that works for $i+1$ also works for $i$.

Secondly, suppose for $i < N$, we have that $a_N \leq c a_i$ for some constant $c > 0$. Then we have 
\begin{equation}
        \sum_{a_i^2 < n \leq \aii^2} \frac{1}{V(\sqrt{n})} 
        \leq C_1 (\aii^2-a_i^2) \frac{\aint{N}}{\aint{N-i} a_i^i}
        \leq C_2 \aii^2 \leq C_2 \aint N.
        \label{eq:rect2}
\end{equation}

Now, we estimate 
$H(x,y) \asymp \sum_{n = 0}^{a_N^2} \frac{1}{V(\sqrt{n})}$
depending on the relationship between $\a \leq \ai \leq \aii$: 

\begin{description}
    \item[Case 1 ($(\aii \leq 2 \ai) \, \& \, (\ai \leq 2\a) $) ]
    In this case, we know that all three of the largest sides are comparable in size. By (\ref{eq:rect2}), we have 
    \begin{align*}
        \sum_{\a^2 < n \leq \aii^2} \frac{1}{V(\sqrt{n})} 
        \leq C_3 \aint N.
    \end{align*}
    Combined with (\ref{eq:rect-a}), we have $H(x,y) \asymp A_{[0, N]}$, which is comparable to the left hand side of (\ref{eq:rect}) in this case. 
    \item[Case 2 ($(\aii \leq 2 \ai) \, \& \, (\ai \geq 2\a)$) ]
    First, by (\ref{eq:rect2}), we have 
    \begin{align*}
        \sum_{\ai^2 < n \leq \aii^2} \frac{1}{V(\sqrt{n})} 
        \leq C_4 \aint N.
    \end{align*}
    Then, 
    \begin{align*}
    \sum_{\a^2 < n \leq \ai^2} \frac{1}{V(\sqrt{n})}
    \asymp a_N a_{N-1} \sum_{\a^2 < n \leq \ai^2} \frac{1}{n}
    \asymp a_N a_{N-1} \log \left( \frac{\ai}{\a} \right).
    \end{align*}
    Combined with (\ref{eq:rect-a}), we have $H(x,y) \asymp A_{[0, N]} + a_N a_{N-1} \log \left( \frac{\ai}{\a} \right)$, giving the desired estimate. 
    \item[Case 3 ($(\aii \geq 2 \ai) \, \& \, (\ai \leq 2\a) $) ]
    By (\ref{eq:rect1}), we have 
    \begin{align*}
        \sum_{\a^2 < n \leq \aii^2} \frac{1}{V(\sqrt{n})} 
        \asymp \aii \sum_{\a^2 < n \leq \aii^2} \frac{1}{\sqrt n}
        \asymp \aii (\aii - \a) \asymp \aii^2,
    \end{align*}
    where the last asymptotic is by $2\a \leq  \aii$. 
     Combined with (\ref{eq:rect-a}), we have $H(x,y) \asymp A_{[0, N]} + a_N^2)$.
    \item[Case 4 ($(\aii \geq 2 \ai) \, \& \, (\ai \geq 2\a) $) ]
    Similar to the computation in the previous cases, we have
    \[
    \sum_{\a^2 < n \leq \ai^2} \frac{1}{V(\sqrt{n})} 
    + \sum_{\ai^2 < n \leq \aii^2} \frac{1}{V(\sqrt{n})}
        \asymp a_N a_{N-1} \log \left( \frac{\ai}{\a} \right) + a_N^2. 
    \]
     Combined with (\ref{eq:rect-a}), we have the desired estimate.
\end{description}

\end{proof}

\subsection{Spaces with Ahlfors regularity} 
\begin{defn}
    A finite graph $\Gamma=(\V,E)$ equipped with a probability measure $\pi$ satisfies 
    \emph{the Ahlfors regularity condition} if there exists $\alpha > 0$ such that 
    \(|\V| V(x,r)\asymp r^\alpha\) 
    for all $x\in \V$ and all $1\le r\le D$, where $D$ is the diameter of $\Gamma$.
\end{defn}
For such a space, $|\V| \asymp D^\alpha$ and $\pi(x)\asymp D^{-\alpha}$. 
Furthermore, 
\begin{equation}\label{Ahlfors3cases} \sum_{n=0}^{2D^\theta}\frac{1}{V(x,n^{1/\theta})}\asymp D^{\alpha} \sum_{n=1}^{2D^\theta} n^{-\frac{\alpha}{\theta}} \asymp \begin{cases}
D^{\alpha} &\text{if $\alpha> \theta$},\\
D^\theta \log D   &\text{if $\alpha = \theta$},\\
D^\theta          &\text{if $0\le \alpha<\theta$}. \end{cases}\end{equation}
Assume that  $(\Gamma,K,\pi)$ satisfies \assumptions{}, and is Ahlfors regular. Then, for any two points $a,b$  in $\V$ with  $d(a,b)\asymp D$, we have
$$H(a,b) \asymp \begin{cases}
D^{\alpha} &\text{if $\alpha> \theta$},\\
D^\theta \log D   &\text{if $\alpha = \theta$},\\
D^\theta          &\text{if $0\le \alpha<\theta$}. \end{cases}$$

\subsection{Doubling spaces with \texorpdfstring{$\theta$}{theta}-fast volume growth}  
A finite graph $(\Gamma, \pi)$ has $\theta$-fast volume growth if there exists $\epsilon$ such that the volume function of $(\Gamma,K,\pi)$ satisfies
$$\frac{V(x, r)}{\pi(x)}\ge c_0 r^{\theta+\epsilon}, \;\;1/2\le r\le D.$$
In such a case,
$$\sum_{n=0}^{2D^\theta}\frac{1}{V(x,n^{1/\theta})}
=\frac{1}{\pi(x)}\sum_{n=0}^{2D^\theta}\frac{\pi(x)}{V(x,n^{1/\theta})} 
\asymp \frac{1}{\pi(x)}.$$
Assume that  $(\Gamma,K,\pi)$ satisfies  \assumptions{} and has $\theta$-fast volume growth, for any two points $a,b$  in $\V$ with  $d(a,b)\asymp D$, we have
$$H(a,b)\asymp \frac{1}{\pi(b)}.$$
Under these hypotheses, one can in fact estimate $H(a,b)$ for all $a,b\in \V$, not just those pairs with $d(a,b)\asymp D$.  Indeed, we have $\G(b,b)\asymp 1/\pi(b)$ and, by Proposition \ref{prop:greenupper}, 
$$|\G(a,b)|\le  C\sum_{d(a,b)^\theta \leq n \leq 2D^\theta}\frac{1}{V(b,n^{1/\theta})}\le C'\frac{d(a,b)^\theta}{V(x,d(a,b))}\le C''\frac{1}{\pi(b)} \frac{1}{d(a,b)^\epsilon}.$$
Hence,
$$H(a,b)=\G(b,b)-\G(a,b)\ge \frac{1}{\pi(b)}\left( c-\frac{C''}{d(a,b)^\epsilon}\right).$$
This suffices to prove the following result.
\begin{prop} Assume $(\Gamma,K,\pi)$ satisfies \assumptions{}, and has $\theta$-fast volume growth. Then, for $a,b\in \V$,
$$H(a,b)\asymp \frac{1}{\pi(b)}.$$
\end{prop}

A rectangular torus 
\(
\Gamma = \Z_{a_1} \times \Z_{a_2} \times \dotsb \times \Z_{a_N} 
\) 
with $a_1\le \dots\le a_N$ has $2$-fast volume growth if and only if $N\ge 3$ and $a_N\approx a_{N-1}\approx a_{N-2}$.
For another example, consider the Cayley graph of the finite Heisenberg of all $3\times 3$ upper-triangular matrices with entries in $\mathbb Z/N\mathbb Z$ and entries equal to $1$ on the diagonal
with generators 
\[
\raisebox{0.5\depth}{$\begin{pmatrix} 1&\pm 1&0\\0&1&0\\0&0&1\end{pmatrix}$},
\raisebox{0.5\depth}{$\begin{pmatrix} 1&0&0\\0&1&\pm 1\\0&0&1\end{pmatrix}$}
\] 
equipped with its natural simple random walk. This example satisfies
$(\El), (\VD), (\PI(2))$, and $(\CS(2))$,
has $2$-fast volume growth and $\pi(x)\equiv 1/N^3$. 
It is not Ahlfors regular.  See \cite[Lemma 4.1]{diaconisModerateGrowthRandom1994}
and note that $\theta$-fast volume growth is different from the notion of moderate growth introduced in that paper.

\subsection{Traces on \texorpdfstring{$\Z^2$ and $\Z^3$}{the 2d and 3d integer lattice}
}
Let $\alpha \in (0,1]$ and $N \geq 5$. 
Consider the subgraph of $\Z^2$ that is traced by the area bounded by the curves
$y = \pm x^\alpha$ and $x = N$. More specifically, define 
\begin{equation*}
\V = \left 
\{ (x,y) \in \Z^2: 
  0 \leq x \leq N          \quad \& \quad
  y \leq x^\alpha           \quad \& \quad 
  y \geq -x^\alpha 
\right \},
\end{equation*}
\noindent and $\Gamma$ be the induced subgraph of $\Z^2$ by this vertex set. 
See Figure \ref{fig:trace} for an example of such a graph. 
\begin{figure}
  \centering
  \includegraphics[width=.35\textwidth]{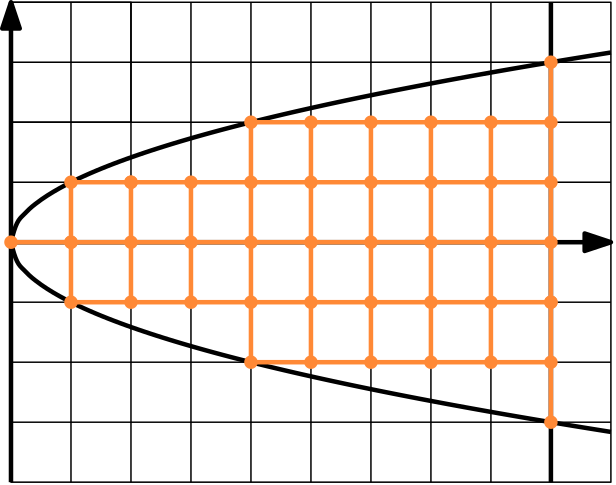}
  \caption{The trace graph with $\alpha = 1/2$ and $N=9$ shown in {\color{kanzou}orange} is a finite subgraph $\Z^2$.}
  \label{fig:trace}
\end{figure}

This graph satisfies ($\VD$) by inspection (see below). 
Moreover, it is known that subgraphs of $\Z^2$ traced by convex sets satisfy $(\PI(2))$, 
see \cite[Section 6]{diaconisNashInequalitiesFinite1996}.
Let $o = (0,0)$ and $p = (N,0)$, which are roughly diameter of the graph apart. 
Then we have 
\begin{equation*}
  \# B( o, r) \asymp r^{1+\alpha} \and \# B( o, n^{1/2}) \asymp n^{\frac{1+\alpha}2}.
\end{equation*}
By Theorem \ref{thm:main}, we have 
\[
  H(p,o) \asymp N^{1 + \alpha} \sum_{n =1}^{N^2} n^{- \frac{1+\alpha}{2}} \asymp
  \begin{cases}
    N^2 \log N &\text{if } \alpha = 1 \\
    N^2 &\text{if } \alpha \in (0,1).
  \end{cases}
\]
For the walk started at $o$, the volume growth and expected hitting time are 
\begin{align*}
  \# B(p, r) &\asymp 
  \begin{cases}
    r^2             &\text{if } r \in (0,N^\alpha)\\
    r N^\alpha       &\text{if } r \in (N^\alpha, N).
  \end{cases}
\and 
H(o,p) \asymp
  \begin{cases}
    N^2 \log N &\text{if } \alpha = 1 \\
    N^2 &\text{if } \alpha \in (0,1).
  \end{cases}
\end{align*}

Now, in $\mathbb R^3$, consider the solid body around the positive semi-axis enclosed by the surface of revolution obtained by rotating  $y = x^\alpha ( \log (1+ x))^\beta$ about the $x$-axis when $\alpha\in (0,1)$ and $\beta\in \mathbb R$. Let $\V$ be the trace of this domain in $\Z^3$, and $\Gamma = (\V,E)$ be the induced subgraph in $\Z^3$. Set $o = (0,0,0)$ and $p = (N,0,0)$. Again, one can check that ($\VD$) and ($\PI(2)$) are satisfied for any of these graphs.  Volume from $o$ is
$$\#B(o,r)\asymp r^{1+2\alpha}(\log (1+ r))^{2\beta}.$$
Using Theorem \ref{thm:main}, we get (assuming $N\ge 2$)
\[
H(p,o) = \begin{cases}
N^2                  &\text{if } \alpha \in (0,1/2) \\
N^{1+2\alpha} (\log N)^{2\beta}    &\text{if } \alpha > 1/2.
\end{cases}
\]
In the omitted border case when $\alpha=1/2$, the parameter $\beta$ plays a role and
$$\sum_{n=0}^{N} \frac{1}{V(o,\sqrt{n})}
\asymp N^{2}(\log N)^{2\beta} \sum_{n=1}^N\frac{1}{n(\log (1+ n))^{2\beta}} 
\asymp \begin{cases}
N^2(\log N)^{2\beta} &\text{if } \beta \in (1/2,+\infty) \\
N^{2} (\log N)(\log\log N)   &\text{if } \beta = 1/2
\\
N^{2} \log N   &\text{if } \beta \in (-\infty, 1/2).
\end{cases}
$$

\subsection{Birth-death chains}

Birth-death chains are a classical example of Markov chains representing the evolution of population over time. 
The state space is $\{0,1,\dotsc,N\}$ and transition probabilities can be specified by the information $\{(p_k, r_k, q_k)\}_{k=0}^N$, with $q_0 = p_N = 0$.
Then the kernel driving the chain is
\[
K(i,j) = 
\begin{cases}
p_i &\text{$j = i+1$} \\
q_i &\text{$j = i-1$} \\
r_i &\text{$j = i$} \\
0   &\text{otherwise}.
\end{cases}
\]
The reversible measure for this chain is proportional to $\prod_0^{N-1}(p_i/q_{i+1})$.
Conversely, given a positive weight function on the state space, $w_k$, $k\in \{0,\dots,N\}$, we can choose appropriate $p_k, r_k$ and $q_k$ so that the Markov chain is reversible with reversible measure proportional to $w_k$,
\(
\pi(k) = \frac{w_k}{\sum_{i=0}^N w_k }.
\)
For instance  \(\pi(i) p_i=\pi(i+1)q_{i+1}=\frac{1}{2}\min\{\pi(i),\pi(i+1)\}\) work (this is the Metropolis chain  for  \(\pi\) with symmetric simple random walk proposal). See, e.g., \cite{nashex}.

Consider the chain with weight function $w_k = (1+k)^\alpha$ for some $\alpha > 0$. 
It is clear that this chain satisfies $(\VD)$, and from \cite{nashex}, this chain also satisfies $(\PI(2))$. 
Using Theorem \ref{thm:main}, we compute
\begin{align*}
V(0,r) \asymp \left( \frac rN \right)^{1+\alpha} 
\and 
H(N,0) \asymp N^{1+\alpha} \sum_{n=1}^{N^2} n^{-\frac{\alpha+1}2} \asymp \begin{cases}
N^{\alpha+1} &\text{if $\alpha>1$},\\
N^2 \log N   &\text{if $\alpha = 1$},\\
N^2          &\text{if $0\le \alpha<1$}.
\end{cases}
\end{align*}
However,
\begin{align*}
V(N,r) \asymp \frac{rN^\alpha}{N^{1+\alpha}} = \frac{r}{N}
\and 
H(0,N) \asymp N \sum_{n=1}^{N^2} n^{-1/2} \asymp N^2.
\end{align*}
Note that in this case, $H(0,N)$ and $H(N,0)$ have different orders of magnitude when $\alpha \ge 1$.  It is harder to go from $N$ to $0$ than it is to go from $0$ to $N$ (the chain has a small tendency to push the walker towards $N$).

\begin{figure}
  \centering
     \begin{subfigure}[b]{0.2\textwidth}
         \centering
         \includegraphics[width=\textwidth]{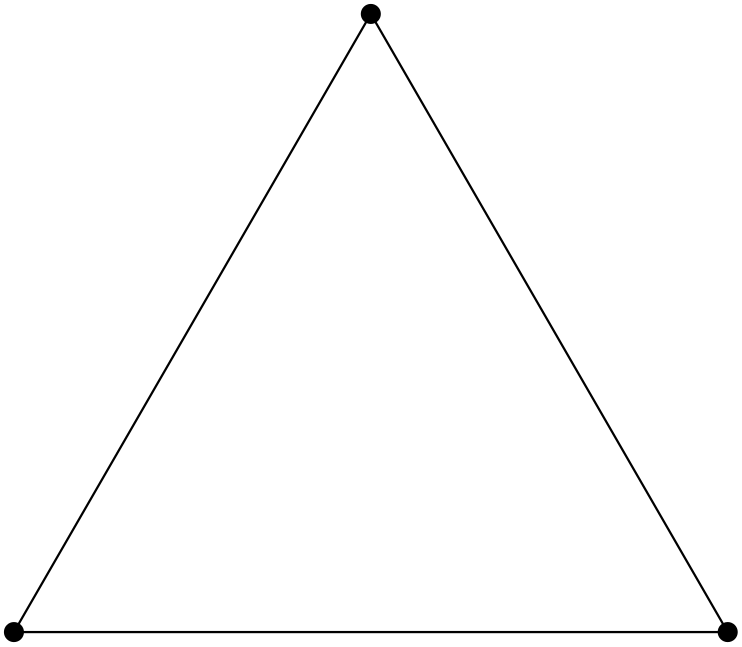}
         \caption{$k=1$}
     \end{subfigure}
     \hfill
     \begin{subfigure}[b]{0.2\textwidth}
         \centering
         \includegraphics[width=\textwidth]{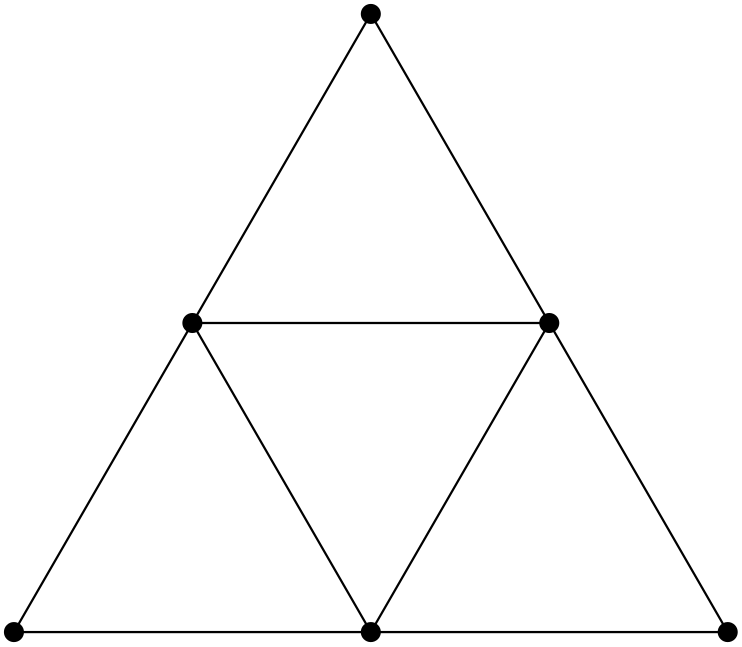}
         \caption{$k=2$}
     \end{subfigure}
     \hfill
     \begin{subfigure}[b]{0.2\textwidth}
         \centering
         \includegraphics[width=\textwidth]{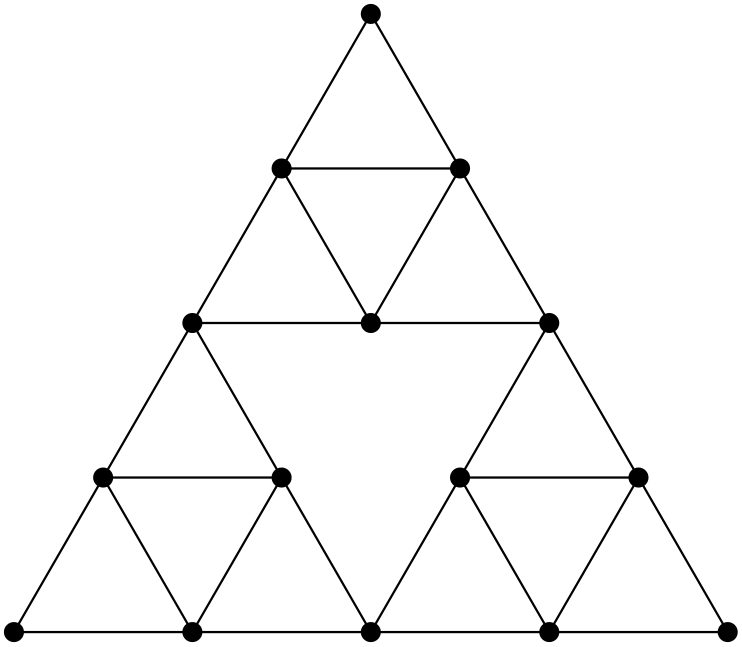}
         \caption{$k=3$}
     \end{subfigure}
    \hfill
     \begin{subfigure}[b]{0.2\textwidth}
         \centering
         \includegraphics[width=\textwidth]{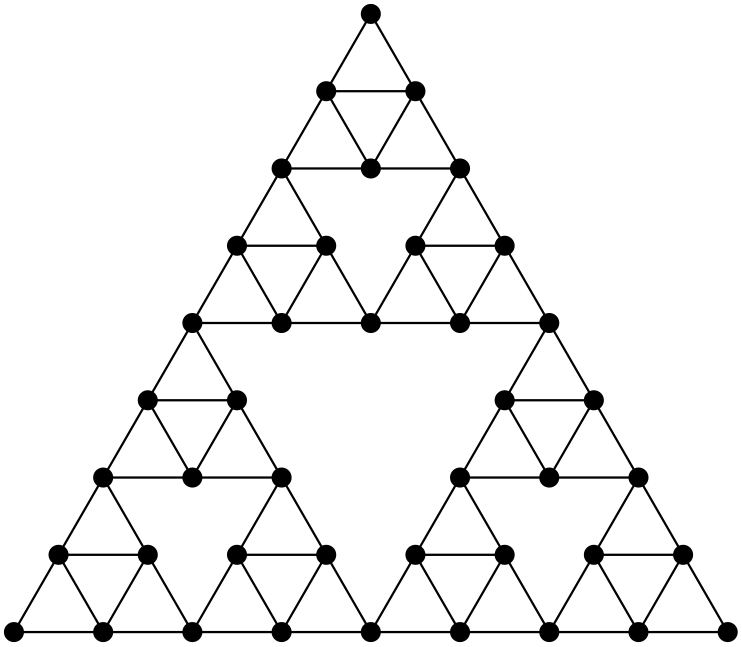}
         \caption{$k=4$}
     \end{subfigure}
  \caption{The Sierpinski triangle $\st_k$ for various $k$.}
  \label{fig:sg}
  \vspace{4mm}
     \begin{subfigure}[b]{0.2\textwidth}
         \centering
         \includegraphics[width=\textwidth]{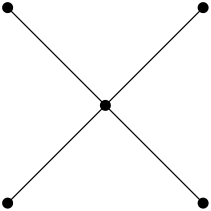}
         \caption{$k=1$}
     \end{subfigure}
     \hfill
     \begin{subfigure}[b]{0.2\textwidth}
         \centering
         \includegraphics[width=\textwidth]{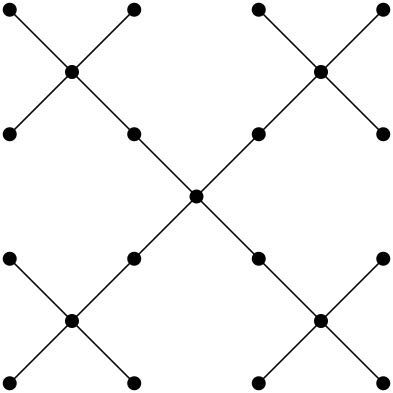}
         \caption{$k=2$}
     \end{subfigure}
     \hfill
     \begin{subfigure}[b]{0.2\textwidth}
         \centering
         \includegraphics[width=\textwidth]{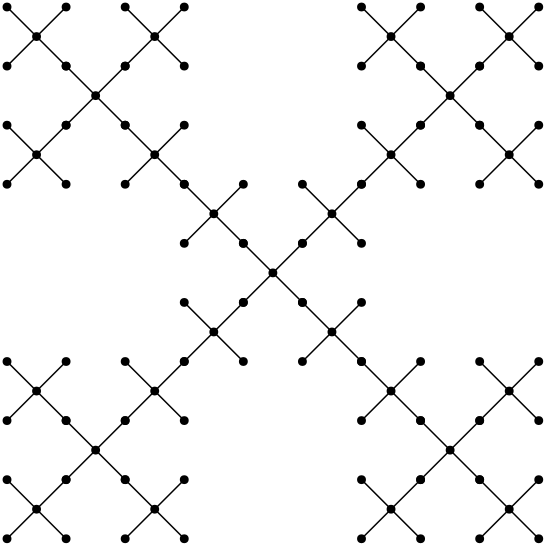}
         \caption{$k=3$}
     \end{subfigure}
    \hfill
     \begin{subfigure}[b]{0.2\textwidth}
         \centering
         \includegraphics[width=\textwidth]{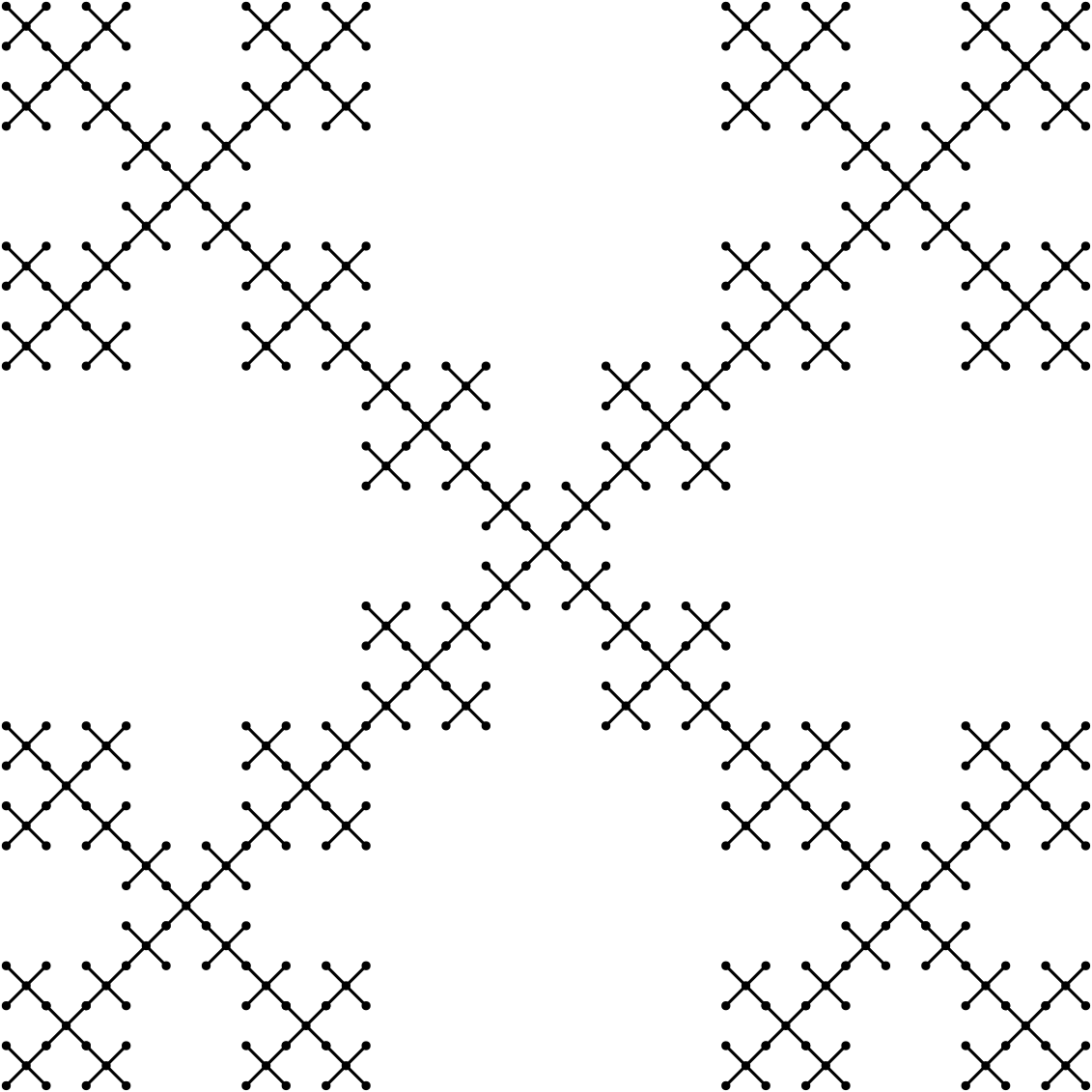}
         \caption{$k=4$}
     \end{subfigure}
  \caption{The Vicsek fractal $\vi_k$ for various $k$.}
  \label{fig:vic}
\end{figure}

\subsection{Fractal graphs}
In the work of \cite{barlowCharacterizationSubGaussianHeat2005}, the authors show that a large class of fractal graphs satisfies the desired criteria for Theorem \ref{thm:main}, with $\theta$ being the walk dimension of the fractal. 
Here we will show how to use our main theorem to compute expected hitting times on two fractal graphs.  These examples are example of Ahlfors-regular graphs.

Let $\st_k$ be the $k^{th}$ iteration of the Sierpinski triangle. $\st_1$ is the complete graph with three vertices, a triangle graph. 
At the $k^{th}$ iteration, we use three copies of $\st_{k-1}$ and glue them at the corners, see Figure \ref{fig:sg}. 
$\st_k$ has size $|\V| = 3^k$, diameter $D \asymp 2^k$, and walk dimension $\theta = \log 5/\log 2$. 
Let $o$ and $p$ be two corners of $\st_k$. 
It is known that
\begin{equation*}
  \# B(o,r) \asymp \min(r^\alpha, 3^k), 
\end{equation*}
where $\alpha = \log 3 /\log 2$.
Applying Theorem \ref{thm:main}, we get 
\begin{equation*}
  H(o,p) \asymp 3^k \sum_{n = 1}^{2^{k \theta}} n ^{-\frac{\log 3}{\log 5}}
  \asymp 3^k \left( 5^k \right) ^{1-\frac{\log 3}{\log 5} } 
  \asymp 5^k = D^\theta, \;\;\theta=\frac{\log 5}{\log 2}.
\end{equation*}

Let $\vi_k$ be the $k^{th}$ iteration of the Vicsek fractal graph. 
$\vi_0$ is a tree with one root vertex with four children. 
From $\vi_{k-1}$, the next iteration $\vi_k$ has $\vi_{k-1}$ joined at the corners with four more copies of $\vi_{k-1}$, see Figure \ref{fig:vic}.
$\vi_k$ has size $|\V| = 5^k$, diameter $D \asymp 3^k$, and walk dimension $\theta = 1+ \alpha = 1 + \log 5/\log 3$. 
Let $o$ and $p$ be two corners of $\vi_k$. 
We know that $\vi_k$ has volume growth
\begin{equation*}
  \# B(o,r) \asymp \min(r^\alpha, 5^k) ,
\end{equation*}
where $\alpha = \log 5/\log 3$.
Our main theorem gives 
\begin{equation*}
  H(o,p) \asymp 5^k \sum_{n = 1}^{3^{k \theta}} n ^{-\frac{\log 5}{\log 3+ \log 5}} 
  \asymp 5^k \left( (3 \cdot 5)^k \right)^{ \frac{\log 3}{\log 3+ \log 5}} 
  \asymp 15^k = D^\theta, \;\;\theta=1+\frac{\log 5}{\log 3}.\end{equation*}

The Vicsek example has a wide ranging extension to finite graphs that are Ahlfors-regular trees as such graphs are Harnack graphs (see \cite{barlowCharacterizationSubGaussianHeat2005}). Experts on analysis on fractals know that there are $\theta$-Harnack fractal graphs with volume growth illustrating all the three cases in (\ref{Ahlfors3cases}) but such examples are not easily available yet in the literature. There is no doubt that examples of $\theta$-Harnack graphs with more general doubling volume behavior exists as well. One major difficulty in obtaining such examples is verifying condition $(\CS(\theta))$.

\subsection{Relaxation times for random walks on lamplighter groups}
Given a finite graph $\Gamma = (\V, E)$, the lamplighter graph $\Gamma^{\diamond} = \mathbb{Z}_2 \wr G$ has the vertex set $\V^{\Gamma^\diamond} = \{0,1\}^\V \times \V$.
Given $(f,v) \in \V^{\Gamma^\diamond}$, $f$ represents the lamp configuration and $v$ the position of the lamplighter. 
For a random walk on $\Gamma$ driven by a Markov kernel $K$, we can define random walk on $\Gamma^\diamond$ was follows
\begin{equation*}
K^\diamond ( (f,v), (h,w)) = \begin{cases}
K(v,w)/4  & \text{if ($v\neq w$) \& ($f$ and $h$ agree off of $\{v,w\}$)} \\
K(v,v)/2  & \text{if ($v=w$) \& ($f$ and $h$ agree off of $v$)} \\
0         & \text{otherwise.}
\end{cases}
\end{equation*}
For an introduction to this topic, see \cite[Chapter 19]{levinMarkovChainsMixing2017}.
There are many interesting connections between the behavior of the random walk on $G^\diamond$ and the one on the underlining graph $G$, one of which is the following: 

\begin{thm}\cite[Theorem 1.2]{peresMixingTimesRandom2004}
Let $\Gamma$ be a finite vertex-transitive graph. 
Then the simple random walk on $\Gamma^\diamond$ satisfies
\[t_{rel}(\Gamma^\diamond) \asymp \max_{x,y \in \V} H_\Gamma(x,y),\]
where $H_\Gamma(x,y)$ is the expected hitting time from $x$ to $y$ for the simple random walk on $G$.
\end{thm}

Theorem \ref{thm:main} also then gives an estimate for the relaxation time for random walks on lamplighter groups. 

\begin{cor}
Suppose that the simple random walk on a vertex-transitive graph $\Gamma$ satisfies
$(\El), (\VD), (\PI(\theta))$, and $(\CS(\theta))$, for some $\theta \geq 2$.
Then, the associated random walk on $G^\diamond$ satisfies
\begin{equation*}
   t_{rel}(\Gamma^\diamond) \asymp \max_{x \in \V} 
   \sum_{0 \leq n \leq 2D_\Gamma^\theta}\frac{1}{V_\Gamma(x,n^{1/\theta})},
\end{equation*}
where $D_\Gamma$ is the diameter of $\Gamma$, and $V_\Gamma$ is normalized volume of balls in $\Gamma$. 
\end{cor}

\begin{proof}
Let $(\Gamma,K,\pi)$ be the Markov kernel associated with the lazy simple random walk on $\Gamma$, 
and $\G$ be the Green function for the base graph, $\Gamma$. 
Note that the upper bound of Theorem \ref{thm:main} can be applied for all $x,y \in \Gamma$, which gives the upper bound directly. 
For the lower bound, we can restrict to $x,y$ far apart and apply Theorem \ref{thm:main} to get
\[
t_{rel}(\Gamma^\diamond) \geq \max_{x,y: d(x,y) = D_\Gamma} H_\Gamma (x,y) 
\gtrsim \max_{x \in \V} \sum_{0 \leq n \leq 2D_\Gamma^\theta}\frac{1}{V_\Gamma(x,n^{1/\theta})}.
\]
\end{proof}

In addition, given Proposition \ref{prop:rect}, we have the following result for lamplighter groups on rectangular tori:
\begin{cor} 
  Define the graph
  \[
  \Gamma = \Z_{a_1} \times \Z_{a_2} \times \dotsb \times \Z_{a_N}, 
  \]
  where $1 \leq a_1 \leq a_2 \leq \cdots \leq a_N$. 
  Consider corresponding random walk on $\Gamma^\diamond$, we have 
  \begin{equation}
    t_{rel} (\Gamma^\diamond) \asymp_N \max \left\{ \prod_{i=1}^N a_i, a_N a_{N-1} \log \left(\frac{a_{N-1}}{a_{N-2}}\right), a_N^2\right\}.
    \label{eq:rect}
  \end{equation}
\end{cor}

\newpage
\appendix

\section{Lemmas}

\begin{lem}
  Let $C_2 > 0$, $\theta \geq 2$ and the Markov chain $\syst$ satisfies $(\VD)$.
  There exists a $C_0 > 0$ such that for all $x,y \in \V$ and  $t\in [d(x,y),d(x,y)^\theta]$,
  \begin{equation*}
    \frac{1}{V\left(x, t^{1 / \theta}\right)}\exp \left(-C_2 \left(\frac{d(x, y)^\theta}{t}\right)^{1 /(\theta-1)}\right)
    \leq \frac{C_0}{V\left(x,d(x,y)\right)}
  \end{equation*}
  \label{lem:appen1}
\end{lem}

\begin{proof}
  Let $A \geq 1$ be the doubling constant of $\Gamma$.
  Since $\syst$ satisfies $(\VD)$, for $0 \leq m \leq n \leq \infty$, we know that
  \begin{equation}
    \frac{V(x,n)}{V(x,m)} \leq A \left( \frac nm \right)^{\frac{\log A}{\log 2}}.
    \label{eq:doublecomp}
  \end{equation}
  See \cite[Lemma 5.1]{diaconisModerateGrowthRandom1994} for a proof.
  Then for all $d(x,y) \leq t \leq d(x,y)^\theta$,
  \begin{equation}
  \frac{V\left(x,t^{1/\theta}\right)}{V(x,d(x,y))} \geq \frac 1A \left( \frac{t^{1/\theta}}{d(x,y)} \right)^{\frac{\log A}{\log 2}}.
  \label{eq:doubletheta}
  \end{equation}
  Let $z = d(x,y)/t^{1/\theta}$.
  Then, choose $C_0 > 0$ large enough so that
  \[
  \frac{\log A}{\log 2} \log z \leq \log C_0 - \log A + C_2 z^{\frac \theta {\theta -1}}.
  \]
  for all $z \geq 1$.
  We know that this is possible because the derivative of the left hand side is positive decreasing, while that of the right hand side is positive increasing.
  After rearranging and taking exponential on both sides, we get
  \[
  \frac 1A \left( \frac{t^{1/\theta}}{d(x,y)} \right)^{\frac{\log A}{\log 2}}
  \geq C_0 \exp \left( -C_2 \left( \frac{d(x,y)^\theta}{t} \right)^{\frac 1{\theta-1}} \right).
  \]
  Combined with (\ref{eq:doubletheta}), this gives us the desired result.
\end{proof}

\begin{lem}
Let $\syst$ be a Markov chain satisfying $(\VD)$, $x \in \V$ and $c_0 \in (0,1) $.
Then there exists a constant $C > 0$, that depends on $c_0$, such that
\[ V(x,c_0 D) > C. \]
\label{lem:constantvol}
\end{lem}
\begin{proof}
Let $A > 0$ be the doubling constant of $\Gamma$.
Then, we use the same formula as in the previous lemma, (\ref{eq:doublecomp}) with $m = c_0 D$ and $n = D$:
\begin{align*}
\frac{1}{V(x,c_0 D)}
\leq A \left( \frac 1 {c_0} \right)^{\frac{\log A}{\log 2}}.
\end{align*}
\end{proof}



\bibliographystyle{alpha}
\bibliography{hittimes}

\appendix

\end{document}